\documentclass[11pt,oneside, a4paper]{amsart} 

\usepackage[english]{babel}
\usepackage[T1]{fontenc}
\usepackage{amsmath}
\usepackage{amssymb}
\usepackage{float}
\usepackage{amsfonts}
\usepackage{mathtools}
\usepackage[utf8]{inputenc}
\usepackage[mathcal]{eucal}
\usepackage{enumerate}
\usepackage{amsthm}
\usepackage{hyperref}
\usepackage[totalwidth=14cm,totalheight=20cm]{geometry}
\usepackage{epsfig,fancyhdr,color}
\usepackage{multicol} 

\newtheorem{cor}{Corollary}[section]

\newtheorem{teo}[cor]{Theorem}

\newtheorem{prop}[cor]{Proposition}

\newtheorem{lemma}[cor]{Lemma}

\theoremstyle{definition}
\newtheorem{defi}[cor]{Definition}
\theoremstyle{remark}

\newtheorem*{remark*}{Remark}

\newcommand{\Pp}{\mathbb{P}}
\newcommand{\R}{\mathbb{R}}

\newcommand{\C}{\mathbb{C}}
\newcommand{\Z}{\mathbb{Z}}

\newcommand{\h}{\mathbb{H}}
\newcommand{\N}{\mathbb{N}}

\newcommand{\diag}{\mathrm{diag}}

\newcommand{\SL}{\mathrm{SL}}

\newcommand{\dPSL}{\mathbb{P}SL(2,\mathbb{R})\times \mathbb{P}SL(2,\mathbb{R})}
\newcommand{\PSL}{\mathbb{P}SL}
\newcommand{\T}{\mathpzc{T}}

\newcommand{\Isom}{\mathrm{Isom}}

\newcommand{\SO}{\mathrm{SO}}

\newcommand{\AdS}{\mathrm{AdS}}
\newcommand{\Imm}{\mathcal{I}\text{m}}
\newcommand{\Teich}{\mathcal{T}}
\renewcommand{\Re}{\mathcal{R}\text{e}}

\DeclareMathAlphabet{\mathpzc}{OT1}{pzc}{m}{it}

\title[Wild GHM AdS structures]{Wild globally hyperbolic maximal \\ anti-de Sitter structures}

\author{Andrea Tamburelli}

\begin{document}

\begin{abstract}
Let $\Sigma$ be a connected, oriented surface with punctures and negative Euler characteristic. We introduce wild globally hyperbolic anti-de Sitter structures on $\Sigma \times \R$ and provide two parameterisations of their deformation space: as a quotient of the product of two copies of the Teichm\"uller space of crowned hyperbolic surfaces and as the bundle over the Teichm\"uller space of $\Sigma$ of meromorphic quadratic differentials with poles of order at least $3$ at the punctures. 
\end{abstract}

\maketitle

\setcounter{tocdepth}{1}

\tableofcontents

\section*{Introduction}
Globally hyperbolic maximal (GHM) anti-de Sitter three-manifolds are a special class of Lorentzian manifolds that share many similarities with hyperbolic quasi-Fuchsian manifolds. Mess initiated the study of the deformation space $\mathcal{GH}(S)$ of such structures (\cite{Mess}) showing that if $S$ is a closed, oriented surface of genus at least $2$, then $\mathcal{GH}(S)$ is parameterised by two copies of the Teichm\"uller space of $S$. After that, many progress has been made in the understanding of the geometry of these manifolds (\cite{folKsurfaces}, \cite{foliationCMC}, \cite{volumeAdS}, \cite{entropy}): in particular, Krasnov and Schlenker (\cite{Schlenker-Krasnov}) noticed that they behave more like almost-Fuchsian hyperbolic manifolds in the sense that they always contain a unique embedded maximal surface (i.e. with vanishing mean curvature) with principal curvatures in $(-1,1)$. They exploited this fact in order to construct a new parameterisation of $\mathcal{GH}(S)$ by the cotangent bundle of the Teichm\"uller space of $S$ by associating to a GHM anti-de Sitter manifold $M$ the conformal class of the induced metric and the holomorphic quadratic differential that determines the second fundamental form of the maximal surface embedded in $M$. \\

This construction has been later generalised by the author to include non-compact surfaces (\cite{Tambu_regularAdS}, \cite{Tambu_poly}). In particular, if $\Sigma$ is a connected, oriented surface with punctures and negative Euler characteristic, we introduced a special class of globally hyperbolic maximal anti-de Sitter structures on $\Sigma\times \R$ that we called \emph{regular} and are parameterised by the bundle over Teichm\"uller space of $\Sigma$ of meromorphic quadratic differentials with poles of order at most $2$ at the punctures. These manifolds play a role also in the theory of GHM anti-de Sitter structures with closed Cauchy-surfaces, as they can be seen as the geometric limits of such structures along pinching sequences in the cotangent bundle parameterisation (\cite{Tambu_pinching}). \\

In this paper we extend our previous results in order to include higher order poles. We expect this theory to be relevant for the study of degeneration of GHM anti-de Sitter structures along more general diverging sequences (\cite{Gupta_limits}). \\

We first show existence and uniqueness of the maximal surface with given embedding data:\\

{\bf Theorem A.} {\it \ Given a complete hyperbolic metric $h$ of finite area on $\Sigma$ and a meromorphic quadratic differential $q$ with poles of order at least $3$ at the punctures, there exists a unique (up to global isometries) complete, conformal equivariant maximal embedding $\tilde{\sigma}:\tilde{\Sigma}\rightarrow \AdS_{3}$ into anti-de Sitter space whose second fundamental form is the real part of $q$.} \\

The embedding $\tilde{\sigma}$ comes together with a representation $\rho:\pi_{1}(\Sigma)\rightarrow \Isom(\AdS_{3})$ that, by identifying $\Isom(\AdS_{3})$ with $\dPSL$, is equivalent to a pair of representations $\rho_{l,r}:\pi_{1}(\Sigma)\rightarrow \PSL(2,\R)$. By the recent work of Gupta (\cite{Gupta}), we will deduce that $\rho_{l,r}$ are faithful and discrete and send peripheral curves to hyperbolic elements. The main part of the paper is devoted to the study of the boundary at infinity of the maximal surface, that, unlike the closed case, is only partially determined by the representation $\rho$. Recall that the boundary at infinity of anti-de Sitter space can be identified with $\R\Pp^{1}\times \R\Pp^{1}$ and the action of $\rho=(\rho_{l},\rho_{r})$ extends naturally on each factor.\\

{\bf Theorem B.}{\it \ The boundary at infinity of $\tilde{\sigma}(\tilde{\Sigma})$ is a locally achronal curve that contains the closure of the set of pairs of attracting fixed points of $(\rho_{l}, \rho_{r})$. This set is completed to a topological circle by inserting, in a $\rho$-equivariant way, a light-like polygonal curve at each end.}\\

We will define precisely in Section \ref{sec:boundary} what we mean by light-like polygonal curve. Here it suffices to mention that it consists of an infinite family of light-like segments on the boundary at infinity of $\AdS_{3}$ belonging to the right-foliation and the left-foliation in an alternate way, which is equivariant by the action of the cyclic group generated by the hyperbolic translation along the corresponding peripheral curve.\\

The boundary at infinity of $\tilde{\sigma}(\tilde{\Sigma})$ determines then a domain of dependence on which $\rho(\pi_{1}(\Sigma))$ acts properly discontinuously and the quotient gives the desired \emph{wild} globally hyperbolic anti-de Sitter manifold diffeomorphic to $\Sigma\times \R$. \\

Moreover, using the relation between maximal surfaces and minimal Lagrangian maps, we are able to give an analogue of Mess' parameterisation for wild anti-de Sitter structures. Recall that an orientation preserving diffeomorphism $m:(\Sigma, h)\rightarrow (\Sigma, h')$ between hyperbolic surfaces is minimal Lagrangian if there exists a Riemann surface $X$ and harmonic maps $f:X \rightarrow (\Sigma,h)$ and $f':X \rightarrow (\Sigma,h')$ with opposite Hopf differentials such that $m=f'\circ f^{-1}$. These are in one-to-one correspondence with $(\rho_{l}, \rho_{r})$-equivariant maximal surfaces in anti-de Sitter space via the Gauss map (\cite{bon_schl}): the Riemann surface $X$ is determined by the conformal structure of the maximal surface, $h$ and $h'$ are hyperbolic metrics on $\Sigma$ with holonomy $\rho_{l}$ and $\rho_{r}$ respectively, and the harmonic maps $f$ and $f'$ are the projections of the equivariant Gauss map (that in this Lorentzian context takes value into $\h^{2}\times \h^{2}$). As a consequence of the work of Gupta (\cite{Gupta}), we deduce that in our case the image of the Gauss map is a pair of crowned hyperbolic surfaces and we prove the following:\\
\\
{\bf Theorem C.}{\it \ The deformation space of wild globally hyperbolic maximal anti-de Sitter structures on $\Sigma\times \R$ is parameterised by the quotient of two copies of the Teichm\"uller space of crowned hyperbolic surfaces by the infinite cyclic group generated by the diagonal action of Dehn twists along the boundary curves and relabelling of the boundary cusps.}\\
%\\
%As a by-product, we also obtain the first examples of minimal Lagrangian maps between crowned hyperbolic surfaces:\\
%\\
%{\bf Theorem D.}{\it \ Let $(\Sigma, h)$ and $(\Sigma, h')$ be crowned hyperbolic surfaces with the same number of boundary cusps at each end. Then there exists a minimal Lagrangian map $m:(\Sigma,h) \rightarrow (\Sigma, h')$.} 

\subsection*{Outline of the paper} In Section \ref{sec:background} we recall well-known facts about anti-de Sitter geometry, meromorphic quadratic differentials and crowned hyperbolic surfaces. In Section  \ref{sec:maxsurface} we prove existence and uniqueness of the equivariant maximal embedding starting from the data of a complete hyperbolic metric of finite area on $\Sigma$ and a meromorphic quadratic differential with poles of order at least $3$. The boundary at infinity of this surface is described in Section \ref{sec:boundary}. We prove Theorem C in Section \ref{sec:para}. %The connection with minimal Lagrangian maps is explained in Section \ref{sec:minLag}.

\subsection*{Acknowledgement} The author would like to thank Subhojoy Gupta for answering specific questions about crowned hyperbolic surfaces.

\section{Background material}\label{sec:background} We recall here some well-known facts about anti-de Sitter geometry, (meromorphic) quadratic differentials on Riemann surfaces, and crowned hyperbolic surfaces that will be used in the sequel. Throughout the paper, we will denote with $\overline{\Sigma}$ a closed, connected, oriented surface and with $\Sigma=\overline{\Sigma}\setminus \{p_{1}, \dots, p_{N}\}$ a surface with a finite number of punctures. We will always assume that $\chi(\Sigma)<0$. Moreover, we will denote with $\mathcal{T}(\Sigma)$ the Teichm\"uller space of $\Sigma$, i.e. the space of marked complete hyperbolic structures of finite area on $\Sigma$ up to isotopy. 

\subsection{Anti-de Sitter geometry} Consider the vector space $\R^{4}$ endowed with a bilinear form of signature $(2,2)$
\[
	\langle x,y\rangle= x_{0}y_{0}+x_{1}y_{1}-x_{2}y_{2}-x_{3}y_{3} \ .
\]
We denote 
\[
	\widehat{\AdS}_{3}=\{ x \in \R^{4} \ | \ \langle x , x \rangle=-1 \} \ .
\]
It can be easily verified that $\widehat{\AdS}_{3}$ is diffeomorphic to a solid torus and the restriction of the bilinear form to the tangent space at each point endows $\widehat{\AdS}_{3}$ with a Lorentzian metric of constant sectional curvature $-1$. Anti-de Sitter space is then defined as
\[
	\AdS_{3}=\Pp(\{x \in \R^{4} \ | \ \langle x,x\rangle < 0\})\subset \R\Pp^{3} \ .
\]
The natural map $\pi:\widehat{\AdS}_{3} \rightarrow \AdS_{3}$ is a two-sheeted covering and we endow $\AdS_{3}$ with the induced Lorentzian structure. The isometry group of $\widehat{\AdS_{3}}$ that preserves the orientation and the time-orientation is $\SO_{0}(2,2)$, the connected component of the identity of the group of linear transformations that preserve the bilinear form of signature $(2,2)$. \\

The boundary at infinity of anti-de Sitter space is naturally identified with 
\[
	\partial_{\infty}\AdS_{3}=\Pp(\{ x \in \R^{4} \ | \ \langle x,x\rangle=0\}) \ .
\]
It coincides with the image of the Segre embedding $s:\R\Pp^{1}\times \R\Pp^{1} \rightarrow \R\Pp^{3}$, and thus, it is foliated by two families of projective lines, which we distinguish by calling $s(\R\Pp^{1} \times \{*\})$ the right-foliation and $s(\{*\}\times \R\Pp^{1})$ the left-foliation. The action of an isometry extends continuously to the boundary, and preserves the two foliations. Moreover, it acts on each line by a projective transformation, thus giving an identification between $\Pp\SO_{0}(2,2)$ and $\dPSL$. \\

The Lorentzian metric on $\AdS_{3}$ induces on $\partial_{\infty}\AdS_{3}$ a conformally flat Lorentzian structure. To see this, notice that the map 
\begin{align*}
	F:D \times S^{1} &\rightarrow \widehat{\AdS}_{3} \\
		(z,w) &\mapsto \left( \frac{2}{1-\|z\|^{2}}z, \frac{1+\|z\|^{2}}{1-\|z\|^{2}}w\right)
\end{align*}
is a diffeomorphism, hence $D\times S^{1}$ is a model for anti-de Sitter space if endowed with the pull-back metric
\[
	F^{*}g_{\AdS_{3}}=\frac{4}{(1-\|z\|^{2})^{2}}|dz|^{2}-\left(\frac{1+\|z\|^{2}}{1-\|z\|^{2}}\right)d\theta'^{2} \ .
\]
Therefore, by composing with the projection $\pi:\widehat{\AdS}_{3}\rightarrow \AdS_{3}$, we deduce that $\pi \circ F$ continuously extends to a homeomorphism
\begin{align*}
	\partial_{\infty}F:S^{1}\times S^{1} &\rightarrow \partial_{\infty}\AdS_{3} \\
				(z,w) &\mapsto (z,w)
\end{align*}
and in these coordinates the conformally flat Lorentzian structure is induced by the conformal class $c=[d\theta^{2}-d\theta'^{2}]$. Notice, in particular, that the light-cone at each point $p \in \partial_{\infty}\AdS_{3}$ is generated by the two lines in the left- and right- foliation passing through $p$. 

\subsection{Complete maximal surfaces in $\AdS_{3}$} \label{sec:maxAdS} Let $U\subset \C$ be a simply connected domain. We say that $\sigma:U \rightarrow \AdS_{3}$ is a space-like embedding if $\sigma$ is an embedding and the induced metric $I=\sigma^{*}g_{AdS}$ is Riemannian. The Fundamental Theorem of surfaces embedded in anti-de Sitter space ensures that such a space-like embedding is uniquely determined, up to post-composition by a global isometry of $\AdS_{3}$, by its induced metric $I$ and its shape operator $B:\sigma_{*}TU \rightarrow \sigma_{*}TU$, which satisfy
\[
	\begin{cases} 
		d^{\nabla}B=0 \ \ \ \ \ \ \ \ \ \ \  \ \ \ \ \ \ \ \ \ \ \ \ \ \ \ \text{(Codazzi equation)} \\
		K_{I}=-1-\det(B) \ \ \ \ \ \ \ \ \ \ \ \ \ \text{(Gauss equation)}
	\end{cases}
\]
where $\nabla$ is the Levi-Civita connection and $K_{I}$ is the curvature of the induced metric on $\sigma(U)$. \\

We say that $\sigma$ is a maximal embedding if $B$ is traceless. In this case, the Codazzi equation implies that the second fundamental form $II=I(B\cdot, \cdot)$ is the real part of a quadratic differential $q$, which is holomorphic for the complex structure compatible with the induced metric $I$, in the following sense. For every pair of vector fields $X$ and $Y$ on $\sigma(U)$, we have
\[
	\Re(q)(X,Y)=I(BX,Y) \ .
\]
In a local conformal coordinate $z$, we can write $q=f(z)dz^{2}$ with $f$ holomorphic and $I=e^{2u}|dz|^{2}$. Thus, $\Re(q)$ is the bilinear form that in the frame $\{\partial_{x}, \partial_{y}\}$ is represented by 
\[
	\Re(q)=\begin{pmatrix}
			\Re(f) & -\Imm(f) \\
			-\Imm(f) & -\Re(f)
		\end{pmatrix} \ ,
\]
and the shape operator can be recovered as $B=I^{-1}\Re(q)$.\\

If the induced metric is complete, the space-like condition implies that, identifying $\widehat{AdS}_{3}$ with $D\times S^{1}$ via $F$, the surface is the graph of a $2$-Lipschitz map (\cite[Proposition 3.1]{Tambu_poly}) and its boundary at infinity $\Gamma$ is a locally achronal topological circle in $\partial_{\infty}\AdS_{3}$ (\cite[Corollary 3.3]{Tambu_poly}) such that if two points are causally related, then a light-like segment joining them is entirely contained in $\Gamma$ (\cite[Lemma 1.1]{Tambu_regularAdS}).

\subsection{GHMC anti-de Sitter manifolds}\label{subsec:GHMC} This paper deals with the moduli space of a special class of manifolds locally isometric to $\AdS_{3}$. \\

We say that an anti-de Sitter three-manifold $M$ is Globally Hyperbolic Maximal (GHM) if it contains an embedded, oriented, space-like surface $S$ that intersects every inextensible causal curve in exactly one point, and if $M$ is maximal by isometric embeddings. It turns out that $M$ is necessarily diffeomorphic to a product $S\times \R$ (\cite{MR0270697}). Moreover, we say that $M$ is Cauchy Compact (C) if $S$ is closed. We denote by $\mathcal{GH}(S)$ the deformation space of GHMC anti-de Sitter structures on $S\times \R$. \\

The theory is well-developed when $S$ is closed of genus at least $2$:
\begin{teo}[\cite{Mess}] $\mathcal{GH}(S)$ is parameterised by $\Teich(S)\times \Teich(S)$.
\end{teo}

The homeomorphism is constructed as follows. Given a GHMC anti-de Sitter structure, its holonomy representation $\rho:\pi_{1}(S) \rightarrow \Isom(\AdS_{3})\cong \dPSL$ induces a pair of representations $(\rho_{l}, \rho_{r})$ by projecting onto each factor. Mess proved that both are faithful and discrete and thus define two points in $\Teich(S)$. On the other hand, given a pair of Fuchsian representations $(\rho_{l},\rho_{r})$, there exists a unique homeomorphism $\phi: \R\Pp^{1}\rightarrow \R\Pp^{1}$ such that $\rho_{r}(\gamma)\circ \phi=\phi\circ \rho_{l}(\gamma)$ for every $\gamma \in \pi_{1}(S)$. The graph of $\phi$ defines a curve $\Lambda_{\rho}$ on the boundary at infinity of $\AdS_{3}$ and Mess constructed a maximal domain of discontinuity $\mathcal{D}(\Lambda_{\rho})$ for the action of $\rho(\pi_{1}(S))$, called domain of dependence, by considering the set of points whose (projective) dual space-like plane is disjoint from $\Lambda_{\rho}$. The quotient
\[
	M=\mathcal{D}(\Lambda_{\rho})/\rho(\pi_{1}(S))
\]
is the desired GHMC anti-de Sitter manifold. \\

%Mess introduced also the notion of convex core. This is the smallest convex subset of a GHMC anti-de Sitter manifold $M$ onto which $M$ retracts. It can be concretely realised as follows. If $\rho$ denotes the holonomy representation of $M$ and $\Lambda_{\rho}\subset \partial_{\infty}\AdS_{3}$ is the limit set of the action of $\rho(\pi_{1}(S))$, the convex core is
%\[
%	\mathcal{C}(M)=\mathcal{C}(\Lambda_{\rho})/\rho(\pi_{1}(S)) \ ,
%\]
%where $\mathcal{C}(\Lambda_{\rho})$ denotes the convex-hull in $\AdS_{3}$ of the curve $\Lambda_{\rho}$. If $M$ is Fuchsian (i.e. the left and right representations coincide), the convex core is a totally geodesic surface. Otherwise, it is a three-dimensional domain, homeomorphic to $S\times [0,1]$, the two boundary components being space-like surfaces, endowed with a hyperbolic metric and pleated along measured laminations.\\

Later Krasnov and Schlenker (\cite{Schlenker-Krasnov}) introduced another parameterisation of $\mathcal{GH}(S)$ by the cotangent bundle over $\Teich(S)$, which is what inspired our construction. Let us recall it briefly here. Let $M$ be a GHMC anti-de Sitter manifold. It is well-known that $M$ contains a unique embedded maximal surface $S$ (\cite{foliationCMC}). Lifting $S$ to $\AdS_{3}$, we obtain an equivariant maximal embedding of $\h^{2}$ into $\AdS_{3}$, which is completely determined (up to global isometries of $\AdS_{3}$) by its induced metric and a holomorphic quadratic differential. By equivariance, these define a Riemannian metric $I$ and a holomorphic quadratic differential $q$ on $S$. We can thus define a map
\begin{align*}
	\Psi: \mathcal{GH}(S) &\rightarrow T^{*}\Teich(S) \\
			M &\mapsto (h,q)
\end{align*}
associating to a GHMC anti-de Sitter structure the unique hyperbolic metric in the conformal class of $I$ and the holomorphic quadratic differential $q$.\\

In order to prove that $\Psi$ is a homeomorphism, Krasnov and Schlenker (\cite{Schlenker-Krasnov}) found an explicit inverse. They showed that, given a hyperbolic metric $h$ and a quadratic differential $q$ that is holomorphic for the complex structure compatible with $h$, it is always possible to find a smooth map $v:S\rightarrow \R$ such that $I=2e^{2v}h$ and $B=I^{-1}\Re(2q)$ are the induced metric and the shape operator of a maximal surface embedded in a GHMC anti-de Sitter manifold. This is accomplished by noticing that the Codazzi equation for $B$ is trivially satisfied since $q$ is holomorphic, and thus it is sufficient to find $v$ so that the Gauss equation holds. Now,
\[
	\det(B)=\det(e^{-2v}(2h)^{-1}\Re(q))=e^{-4v}\det((2h)^{-1}\Re(2q))=-e^{-4v}\|q\|_{h}^{2}
\]
and
\[
	K_{I}=e^{-2v}(K_{2h}-\Delta_{2h}v)=\frac{1}{2}e^{-2v}(K_{h}-\Delta_{h}v)
\]
hence the Gauss equation translates into the quasi-linear PDE
\begin{equation}\label{eq:PDE}
	\frac{1}{2}\Delta_{h}v=e^{2v}-e^{-2v}\|q\|_{h}^{2}+\frac{1}{2}K_{h} \ .
\end{equation}
They proved existence and uniqueness of the solution to Equation (\ref{eq:PDE}) on closed surfaces and on surfaces with punctures, when $q$ has simple pole sigularities at the punctures. In \cite{Tambu_regularAdS} an analogous result was obtained for meromorphic quadratic differentials with poles of order at most $2$ at the punctures. In Section \ref{sec:maxsurface}, we will extend this result to include higher order poles and describe the geometry of the associated maximal surface.

\subsection{Meromorphic quadratic differentials}\label{subsec:QD}
Suppose that $\Sigma$ is endowed with a complex structure. A meromorphic quadratic differential $q$ on $\Sigma$ is a $(2,0)$-tensor, locally of the form $q(z)dz^{2}$, where $q(z)$ is a meromorphic function with poles at the punctures $\{p_{1}, \dots, p_{N}\}$. In this paper, we are interested in meromorphic quadratic differentials with poles of order $n\geq 3$ at the punctures. In this case, we can always find a local coordinate chart around the puncture such that
\[
	q(z)dz^{2}=\left( \frac{a_{n}}{z^{n}}+\frac{a_{n-1}}{z^{n-1}}+\cdots+\frac{a_{2}}{z^{2}}\right) dz^{2}
\]
for some coefficients $a_{j} \in \C$. Meromorphic quadratic differentials with poles at points $p_{1}, p_{2}, \dots, p_{N} \in \overline{\Sigma}$ of orders bounded above by $n_{1}, n_{2}, \dots, n_{N} \in \N$ form a vector space over $\C$ of real dimension $d=3|\chi(\overline{\Sigma})|+2\sum_{i}n_{i}$, by the Riemann-Roch Theorem. In particular, the space of meromorphic quadratic differentials with poles of order exactly $n_{1}, \dots, n_{N}$ is parameterised by $\R^{d-N}\times (S^{1})^{N}$.\\

A meromorphic quadratic differential $q$ induces a singular flat metrics $|q|$ on $\Sigma$ that in local coordinates is written as $|q|=|q(z)||dz|^{2}$. The metric has cone singularities of angle $\pi(m+2)$ at a zero of order $m$ of $q$. When poles have order at least $2$, the metric is complete of infinite area, and poles are at infinite distance from any point on the surface. Moreover, Strebel (\cite{Strebel}) described the local picture of the singular flat metric around a pole $p$ of order $n\geq 3$ as a cyclic arrangement of $(n-2)$ half-planes glued along half-lines in their boundaries. These half-planes are constructed as follows (see also \cite{DW}). First, we choose a local coordinate $w$ adapted to the quadratic differential in the sense that
\[
	w^{*}q=\begin{cases} \frac{1}{w^{n}}dw^{2} \ \ \ \ \ \ \ \ \ \ \ \ \ \ \ \ \ \  \ \ \ \text{if $n$ is odd} \\
	\left(\frac{1}{w^{n/2}}+\frac{A}{w}\right)^{2}dw^{2}  \ \ \ \ \ \ \ \ \text{if $n$ is even} \ .
	\end{cases}
\]
Then, for every $k=1, \dots, n-2$, the half-planes are the images of the natural charts
\[
	\Phi_{k}:\h^{2}\rightarrow V=\{ 0< |w| < r \} \subset \Sigma 
\]
defined by the property that $\Phi_{k}^{*}q=d\omega^{2}$. If $p$ is a pole of odd order, these natural coordinates can be written explicitly as
\begin{equation}\label{eq:natural}
	\Phi_{k}(\omega)=\left( \frac{n-2}{2} \right)^{-\frac{2}{n-2}}\exp \left( -\frac{2}{n-2}\log(\omega+iB)+\frac{2k\pi i}{n-2}\right)
\end{equation}
where $B>0$ is big enough to ensure that $\Phi_{k}(\h^{2})\subset V$. For even order poles the above construction needs to be slightly modified: for $\epsilon>0$ small, we consider
\[
	\h^{2}_{\epsilon}=\{ \omega \in \C \ | \ -\epsilon<\arg(\omega)<\pi + \epsilon\ \}  \ .
\]
Notice that there is a constant $\lambda(\epsilon)>1$ depending on $\epsilon$ such that any pair of points $\omega_{1}, \omega_{2} \in \h^{2}_{\epsilon}$ is connected by a path in $\h^{2}_{\epsilon}$ with length bounded above by $\lambda(\epsilon)|\omega_{1}-\omega_{2}|$. The natural coordinates are then defined as a composition 
\[
	\Phi_{k}=\Psi_{k} \circ \phi: \h^{2} \rightarrow \h^{2}_{\epsilon} \rightarrow V 
\]
where $\phi$ and $\Psi_{k}$ are defined as follows. The map $\Psi_{k}$ is given by (\ref{eq:natural}) extended to $\h^{2}_{\epsilon}$ for a suitable choice of $B$ that guarantees that $\Psi_{k}(\h^{2}_{\epsilon})\subset V$. An easy computation shows that
\[
	\Psi_{k}^{*}q=\left( 1+\frac{C}{\omega+iB}\right)^{2}d\omega^{2}
\]
for a constant $C\in \C$ depending only on $A$. Up to increase $B$ further we can assume that 
\[
	\left| \frac{C}{\omega+iB} \right|<\frac{1}{\lambda(\epsilon)} \ \ \  \text{for every $\omega \in \h^{2}_{\epsilon}$} \ .
\]
With this choice, the map $F(\omega)=\omega+C\log(\omega+iB)$ sends $\h^{2}_{\epsilon}$ injectively into a domain in the complex plane containing $\h^{2}+iD$ fo some constant $D>0$ large enough, thus the function $\phi(\omega)=F^{-1}(\omega+iD)$ is well-defined as map $\phi:\h^{2} \rightarrow \h^{2}_{\epsilon}$ and the composition $\Phi_{k}=\Psi_{k}\circ \phi$ is the desired natural coordinate. \\
\indent Moreover, by construction, the union of the images $\Phi_{k}(\h^{2})$ for $k=1, \dots, n-2$ is a punctured neighbourhood of $p$, and two consecutive charts only intersect along a half-ray in their boundary.

\subsection{Crowned hyperbolic surfaces}
A crown $\mathcal{C}$ with $m\geq 1$ boundary cusps is an incomplete hyperbolic surface bounded by a closed geodesic boundary $c$ and a crown end consisting of bi-infinite geodesic $\{\gamma_{i}\}_{i=1}^{m}$ arranged in cyclic order, such that the right half-line of the geodesic $\gamma_{i}$ is asymptotic to the left half-line of the geodesic $\gamma_{i+1}$, where indices are intended modulo $m$. A crown comes equipped with a labelling of the boundary cusps compatible with the cyclic order. \\

A polygonal end $\mathcal{P}$ of a crown is the $\Z$-invariant bi-infinite chain of geodesic lines in $\h^{2}$ obtained by lifting the cyclically ordered collection of geodesics $\{\gamma_{i}\}_{i=1}^{m}$ in $\mathcal{C}$ to its universal cover, where $\mathbb{Z}$ is the group generated by the hyperbolic translation corresponding to the geodesic boundary $c$. Notice that the ideal points of the chain of geodesics of the polygonal end limit to the end point of the axis $\alpha$ of the lift of $c$. \\

The hyperbolic crowns we will consider come with an additional real parameter, the boundary twist, that we associate with the geodesic boundary. In the corresponding polygonal end in the universal cover, this can be thought of as the choice of a marked point on the axis $\alpha$ and the parameter is the signed distance of this point from the foot of the orthogonal arc from the cusp labelled with $"1"$ to $\alpha$. \\ 

Let $S$ denote a compact, oriented surface of genus $\tau \geq 1$ and $b\geq 1$ boundary components. A crowned hyperbolic surface is obtained by attaching crowns to a compact hyperbolic surface with geodesic boundaries by isometries along their closed boundaries. This results in an incomplete hyperbolic metric of finite area on the surface. We denote with $\mathcal{T}(S, m_{1}, \dots, m_{b})$ the Teichm\"uller space of crowned hyperbolic surfaces such that the $i$-th crown has $m_{i}\geq 1$ boundary cusps, for every $i=1, \dots, b$.  In this context the marking is a homeomorphism $f:S \rightarrow X$ sending a neighbourhood of the boundary to the crown end. Two marked hyperbolic surfaces with crowns $(X,f)$ and $(Y,g)$ are equivalent if there is an isometry $i:X \rightarrow Y$ that is homotopic to $g\circ f^{-1}$ via a homotopy that keeps each boundary component fixed, and $g\circ f^{-1}$ does not Dehn-twist around any crown end.

\begin{prop}[Lemma 2.16 \cite{Gupta}] The Teichm\"uller space of crowned hyperbolic surfaces is homeomorphic to $\R^{n}$, where $n=6\tau-6+\sum_{i=1}^{b}(m_{i}+3)$.
\end{prop}

Here is a possible way to give coordinates to $\mathcal{T}(S, m_{1}, \dots, m_{b})$: the first $6\tau-6+3b$ parameters are the familiar Fenchel-Nielsen coordinates on the Teichm\"uller space of surfaces of genus $\tau$ and $b$ geodesic boundaries. Then, after fixing an identification of the universal cover of the surface with boundary with a domain in $\h^{2}$, the marked crowned hyperbolic surface is determined by the end points of the lifts of the boundary cusps in a fundamental domain. In order to keep track also of the twist parameters, we fix, in an equivariant way, a point on the lifts of each geodesic boundary so that the remaining $\sum_{i=1}^{b}m_{i}$ parameters can be defined as follows: $b$ real parameters are given by the signed distance between the base point fixed above and the foot of the geodesic arc exiting from the boundary cusp labelled with "1" intersecting the geodesic boundary orthogonally, and the other $\sum_{i=1}^{b}(m_{i}-1)$ are positive real numbers determined by the relative distance between the intersection points of the geodesic rays emanating from two consecutive cusps and orthogonal to the geodesic boundary.

\section{Construction of the maximal surface} \label{sec:maxsurface}
In the next sections we are going to construct globally hyperbolic maximal anti-de Sitter structures on $\Sigma\times \R$ starting from the data of a complete hyperbolic metric $h$ on $\Sigma$ of finite area and a meromorphic quadratic differential $q$ with poles of order at least $3$ at the punctures. We first find a complete equivariant maximal embedding into $\AdS_{3}$ with induced metric $I=2e^{2v}h$ and second fundamental form $II=2\Re(q)$. We will then describe its boundary at infinity and prove that $\pi_{1}(\Sigma)$ acts by isometries and properly discontinuously on its domain of dependence, thus inducing a globally hyperbolic anti-de Sitter structure on the quotient.\\

Let $h\in \T(\Sigma)$ be a complete hyperbolic metric of finite area on $\Sigma$ and let $q$ be a meromorphic quadratic differential with poles of order at least $3$ at the punctures. Recall that finding an equivariant maximal conformal embedding of $\tilde{\Sigma}$ into $\AdS_{3}$ is equivalent to finding a solution to the quasi-linear PDE (Section \ref{subsec:GHMC})
\[
	\frac{1}{2}\Delta_{h}v=e^{2v}-e^{-2v}\|q\|^{2}_{h}+\frac{1}{2}K_{h} \ .
\]
This is an example of vortex equation, recently studied in the context of Riemann surfaces with punctures in \cite{nie_poles}. We recall here for the convenience of the reader the main steps for the construction of the unique solution and the asymptotic estimates that will be used in the sequel. \\

The main idea consists in choosing another complete background metric $g$ in the same conformal class as $h$ such that 
\[
	1-\|q\|^{2}_{g}+\frac{1}{2}K_{g} \to 0 \ \ \ \ \ \text{at}  \ \ p_{i} \ .
\]
In this way, the function $u:\Sigma \rightarrow \R$ that satisfies $2e^{2v}h=2e^{2u}g$ is the solution of the differential equation
\[
	\frac{1}{2}\Delta_{g}u=e^{2u}-e^{-2u}\|q\|_{g}^{2}+\frac{1}{2}K_{g}
\]
and the assumptions on $g$ guarantees that $u=0$ is an approximate solution in a neighbourhood of the punctures. This metric $g$ is defined as a smooth interpolation between the metric $\frac{1}{2}h$ of constant curvature $-2$ and the flat metric induced by the quadratic differential. More precisely, we introduce a local coordinate $z_{i}$ in a neighbourhood $U_{i}$ of the puncture $p_{i}$ disjoint from the zeros of $q$ and define
\[
	g=\begin{cases}
		|q| \ \ \ \ \ \ \ \  \ \ \ \ \ \ \ \ \ \text{for} \ |z_{i}|<c_{i}  \\
		e^{\varphi_{i}}|dz_{i}|^{2} \ \ \ \ \ \ \ \ \  \text{for} \  c_{i}\leq |z_{i}| \leq C_{i} \\
		\frac{1}{2}h \ \ \ \ \ \ \ \ \ \ \ \ \ \ \ \  \text{for} \ |z_{i}|>C_{i} \ \text{and on} \ \Sigma\setminus U_{i}
	    \end{cases}
\]
for smooth interpolating functions $\varphi_{i}$. Moreover, we can assume that there exists $\delta_{i}>0$ such that $\|q\|_{g}^{2}\geq \delta_{i}$ on $U_{i}$ because $\|q(z_{i})\|_{g}^{2}\to 1$ when $z_{i}\to 0$. \\

\begin{prop}\label{prop:existence} There exists a bounded smooth function $u: \Sigma \rightarrow \R$ satisfying 
\begin{equation}\label{eq:PDEproof}
	\frac{1}{2}\Delta_{g}u=e^{2u}-e^{-2u}\|q\|_{g}^{2}+\frac{1}{2}K_{g} \ .
\end{equation}
\end{prop}
\begin{proof} Let $F(x,u)=e^{2u}-e^{-2u}\|q\|^{2}_{g}+\frac{1}{2}K_{g}$. Since $F$ is an increasing function of $u$, the solution to Equation (\ref{eq:PDEproof}) is guaranteed (\cite[Theorem 9]{wan}) by the existence of two continuous functions $u^{\pm}:\Sigma\rightarrow \R$ such that
\[
	\Delta u^{+}\leq F(u^{+},x), \ \ \ \Delta u^{-}\geq F(u^{-},x) \ \ \ \ \text{and} \ \ \ u^{-}\leq u^{+} \ .
\]
Let us start with the supersolution $u^{+}$. Let $f:\Sigma\rightarrow \R$ be a positive smooth function such that $f(z_{i})=|z_{i}|^{2\alpha_{i}}$ on the neighbourhood $\{|z_{i}|<c_{i}\}$ of the puncture $p_{i}$ for some $\alpha_{i}>0$ to be chosen later. For any $\beta\in \R$ we consider $u^{+}=\beta f$. We claim that it is possible to find $\beta>0$ large enough and $\alpha_{i}>0$ sufficiently small so that $u^{+}$ is a supersolution. In fact, on $V_{i}=\{|z_{i}|<c_{i}\}\subset U_{i}$, we can find a constant $D_{i}>0$ such that $|q|\geq D_{i}|z|^{-2}$ and we have
\begin{align*}
	&\frac{1}{2}\Delta_{g}(\beta |z_{i}|^{2\alpha_{i}})-e^{2\beta|z_{i}|^{2\alpha_{i}}}+ e^{-2\beta|z_{i}|^{2\alpha_{i}}}\|q\|^{2}_{g}-\frac{1}{2}K_{g} \\
	\leq & \ \frac{1}{2}\beta\alpha_{i}^{2}D_{i}|z_{i}|^{2\alpha_{i}}-e^{2\beta|z_{i}|^{2\alpha_{i}}}+e^{-2\beta|z_{i}|^{2\alpha_{i}}} \\
	\leq &\left( \frac{\alpha_{i}^{2}D_{i}}{2}-2 \right) u^{+}+(e^{-u^{+}}-e^{2u^{+}}+2u^{+})+e^{-2u^{+}} \ ,
\end{align*}
which can be made negative, because the term in the middle is always non-positive and we can choose $\alpha_{i}$ small enough and $\beta$ large enough so that the sum of the first and last term is negative. Therefore, $u^{+}$ is a supersolution on $V_{i}$ for every $\alpha_{i}>\alpha_{0}$ and $\beta>\beta_{0}$. Outside $V_{i}$, we do not have control on the curvature of $g$ and on the Laplacian of $f$, but knowing that they are bounded, we can increase $\beta$ so that
\[
	\frac{\beta}{2}\Delta_{g}f-e^{\beta f}+e^{-2\beta f}\|q\|^{2}_{g}-\frac{1}{2}K_{g}\leq 0
\]
because $e^{\beta f}$ grows the fastest when $\beta\to +\infty$. This proves that $u^{+}$ is a supersolution everywhere on $\Sigma$. \\
\indent As for the subsolution, let $w:\Sigma \setminus q^{-1}(0) \rightarrow \R$ be half of the logarithmic density of the flat metric $|q|$ with respect to $g$, that is $e^{2w}g=|q|$. We claim that $w$ is a solution outside the zeros of $q$: in fact,
\begin{align*}
	&\frac{1}{2}\Delta_{g}w-e^{2w}+e^{-2w}\|q\|_{g}^{2}-\frac{1}{2}K_{g}= \\
	&\frac{1}{2}(\Delta_{g}w-K_{g})-e^{2w}+e^{2w}\|q\|_{|q|}^{2}=0
\end{align*}
because $\|q\|_{|q|}=1$ and the first term vanishes because the metric $|q|$ is flat outside the zeros of $q$. Notice that $w$ tends to $-\infty$ at a zero of $q$ and, by our definition of the open sets $U_{i}$ and of the metric $g$, the background metric on $\Sigma$ has constant curvature $-2$ in a small neighbourhood of the zeros of $q$. Since any negative constant is a subsolution where the metric $g$ has constant curvature $-2$, the function
\[
	u^{-}=\begin{cases}
			w \ \ \ \ \ \  \ \ \ \ \ \ \  \ \ \ \ \ \ \ \ \ \ \text{ on $U_{i}$} \\
			\max(w, -B) \ \ \ \ \ \ \ \ \ \text{ on $\Sigma\setminus U_{i}$}
		\end{cases}
\]
for a sufficiently large $B>0$ is a continuous subsolution, being it the maximum of two subsolutions.
\end{proof}

We remark that the resulting metric $I=2e^{2u}g$ is complete because $g$ is complete and $u$ is bounded. Moreover, the subsolution we found implies that $I\geq 2|q|$. Uniqueness follows then from a general result about vortex equations:
\begin{prop}[\cite{nie_poles} Theorem 2.11] For every non-zero holomorphic quadratic differential on $\Sigma$ there exists a unique complete solution to Equation (\ref{eq:PDE}).
\end{prop}

Combining the above results we obtain:
\begin{teo}\label{teo:maxsurface} For any complete hyperbolic metric $h$ on $\Sigma$ of finite area and for any meromorphic quadratic differential $q$ on $\Sigma$ with poles of order at least $3$ at the punctures there exists a unique complete equivariant maximal embedding $\tilde{\sigma}:\tilde{\Sigma}\rightarrow \AdS_{3}$ with induced metric $I$ conformal to $h$ and second fundamental form $II=\Re(2q)$. Moreover, the principal curvatures are in $(-1,1)$.
\end{teo}
\begin{proof} Existence and uniqueness of such embedding follows from the above discussion. Let $\lambda$ be the positive principal curvature of the maximal surface. By definition of $q$, we have
\[
	-\lambda^{2}=\det(B)=e^{-4u}\|q\|^{2}_{g}\to 1
\]
at the punctures. Therefore, $\lambda$ is bounded and a classical fact about maximal surfaces in anti-de Sitter space (\cite[Lemma 3.11]{Schlenker-Krasnov}) implies that $\lambda \in [0,1)$.
\end{proof}

\subsection{Asymptotic estimates} In order to describe the geometry of the maximal surface, we will also need the following precise estimate for the solution $v$ in a neighbourhood of a puncture. Recall that such a neighbourhood is covered by a collection of half-planes, in which the quadratic differential pulls back to $d\omega^{2}$. By an abuse of notation, we will still indicate with $v$ the function such that $2e^{2v}|d\omega|^{2}$ equals the induced metric on $\tilde{\sigma}(\tilde{\Sigma})$ in the $\omega$-coordinates.

\begin{prop}\label{prop:estimates} Let $\omega$ be a natural coordinate for $q$ defined on a standard half-plane in a neighbourhood of a puncture $p$. Then 
\[
	v(\omega)=O\left( \frac{e^{-2\sqrt{2}|\omega|}}{\sqrt{|\omega|}} \right) \ \ \ \text{as $|\omega| \to +\infty$} \ .
\]
\end{prop}
\begin{proof} In a natural coordinate the function $v$ satisfies the PDE
\begin{equation}\label{eq:PDEerror}
	\frac{1}{2}\Delta v=e^{2v}-e^{-2v}
\end{equation}
because $|q(\omega)|=1$ and the background metric is flat.  The subsolution and supersolution in Proposition \ref{prop:existence} also show that $v$ is non-negative and infinitesimal. In particular, we can assume that $v\leq 1$ on every half-plane and we have that $e^{2v}-e^{-2v}\geq 4v$ . The asymptotic estimate will then follow from \cite[Lemma 5.8]{DW} provided we show that the restriction of $v$ to the boundary of a half-plane is integrable. In order to show this, we prove that $v$ is exponentially decaying. Let $\omega$ be a point in the boundary of the half-plane. If $|\omega|$ is sufficiently large, this point is actually contained also in the precedent or the following standard half-plane. In both cases, we can find a constant $c>0$ depending only on the gluing map between the half-planes (hence only on $q$) and a ball of radius $r(\omega)=|\omega|-c$ centered at $\omega$, which is entirely contained in these two coordinate charts. Using a coordinate $\zeta$ in this ball $B_{r(\omega)}$, the function $v$ satisfies the same Equation (\ref{eq:PDEerror}) on this ball. Therefore, the solution of the Dirichlet problem
\[
	\begin{cases} \Delta h=8h \\
				h_{|_{\partial B_{r(\omega)}}}=1
	\end{cases}   
\]
is a supersolution and as such is greater than $v$. It is then well-known that the solution of the above Dirichlet problem is the function
\[
	h(\zeta)=\frac{I_{0}(2\sqrt{2}|\zeta|)}{I_{0}(2\sqrt{2}|r(|\omega|))}
\]
where $I_{0}$ is the modified Bessel function of the first kind (\cite[Formula 9.6.1]{handbook_formulas}). Hence, (\cite[Formula 9.7.1]{handbook_formulas})
\[
	v(\omega)\leq h(0)=O(|\omega|^{\frac{1}{2}}e^{-2\sqrt{2}|\omega|}) \ \ \ \ \ \ \text{as $|\omega| \to +\infty$} \ .
\]
\end{proof}

\section{Description of the boundary at infinity}\label{sec:boundary}
The equivariant maximal embedding $\tilde{\sigma}:\tilde{\Sigma}\rightarrow \AdS_{3}$ comes with a representation $\rho:\pi_{1}(\Sigma):\rightarrow \Pp\SO_{2}(2,2)$ such that
\[
	\tilde{\sigma}(\gamma\cdot x)=\rho(\gamma)\tilde{\sigma}(x) \ \ \ \ \ \forall x\in \tilde{\Sigma} \ \ \forall \gamma \in \pi_{1}(\Sigma) \ .
\]
Identifying $\Pp\SO_{0}(2,2)$ with $\dPSL$, $\rho$ determines and is determined by a pair of representations $\rho_{l,r}:\pi_{1}(\Sigma)\rightarrow \PSL(2,\R)$. In order to understand them, we make use of the theory of harmonic maps between surfaces. Infact, the maximality of $\tilde{\sigma}(\tilde{\Sigma})$ implies that the Gauss map $G:\tilde{\Sigma}\rightarrow \h^{2}\times \h^{2}$ is harmonic and $(\rho_{l},\rho_{r})$-equivariant. Moreover, the bound on the principal curvatures ensures that, if we denote with $\pi_{l}$ and $\pi_{r}$ the two projections onto the left and right factors, the maps $(G\circ \pi_{l})$ and $(G\circ \pi_{r})$ are harmonic diffeomorphisms into their image (\cite{bon_schl}, \cite[Lemma 6.1]{Tambu_regularAdS}). Then, the hyperbolic metrics
\[
	(G\circ \pi_{l})^{*}g_{\h^{2}} \ \ \ \text{and} \ \ \ (G\circ \pi_{r})^{*}g_{\h^{2}} 
\]
descend to hyperbolic metrics $h_{l}$ and $h_{r}$ on $\Sigma$ with holonomy $\rho_{l}$ and $\rho_{r}$, respectively. Now, since the Hopf differentials of these harmonic maps are $\pm 2iq$, where $\Re(2q)$ is the second fundamental form of the maximal embedding (\cite{Schlenker-Krasnov}, \cite{Tambu_poly}), a recent result by Gupta (\cite[Theorem 1.2]{Gupta}) implies that $(\Sigma, h_{l})$ and $(\Sigma, h_{r})$ are crowned hyperbolic surfaces, thus showing that the representations $\rho_{l}$ and $\rho_{r}$ are discrete and faithful, with hyperbolic peripheral elements. \\

Since the maximal surface is complete, its boundary at infinity is a locally achronal curve $\Gamma$ and determines a domain of dependence $\mathcal{D}(\Gamma)\subset \AdS_{3}$ by considering points whose dual space-like plane is disjoint from $\Gamma$, on which the representation $\rho=(\rho_{l}, \rho_{r})$ acts properly discontinuously. Our knowledge on the representation $\rho$ allows us to describe the curve $\Gamma$ at least partially. Recall that we can identify the boundary at infinity of $\AdS_{3}$ with $\R\Pp^{1}\times \R\Pp^{1}$ and $\rho_{r,l}$ act on each factor by projective transformations. Given an element $\gamma \in \pi_{1}(\Sigma)$, let $x_{\bullet}^{\pm}(\gamma)$ denote the attracting and repelling fixed points of $\rho_{\bullet}(\gamma)$. These define four points on the boundary at infinity of $\AdS_{3}$:
\begin{align*}
	x^{++}(\rho(\gamma))&=(x_{l}^{+}(\gamma), x_{r}^{+}(\gamma)) \ \ \  &x^{+-}(\rho(\gamma))=(x_{l}^{+}(\gamma), x_{r}^{-}(\gamma)) \\
	x^{--}(\rho(\gamma))&=(x_{l}^{-}(\gamma), x_{r}^{-}(\gamma)) \ \ \ \ \  &x^{-+}(\rho(\gamma))=(x_{l}^{-}(\gamma), x_{r}^{+}(\gamma))  
\end{align*}
It follows immediately from the definition that 
\[
	\lim_{n\to +\infty}\rho(\gamma)^{n} \cdot x=x^{++}(\rho(\gamma))
\]
for every $x \in \partial_{\infty}\AdS_{3} \setminus \{ x^{+-}(\rho(\gamma)), x^{-+}(\rho(\gamma)), x^{--}(\rho(\gamma))\}$. Therefore, the limit set 
\[
	\Lambda_{\rho}=\overline{\{ x^{++}(\gamma(\rho)) \in \partial_{\infty}AdS_{3} \ | \ \gamma \in \pi_{1}(\Sigma) \}} 
\]
is the smallest closed $\rho(\pi_{1}(\Sigma))$-invariant subset in the boundary at infinity of anti-de Sitter space. Since the boundary of the maximal surface is $\rho(\pi_{1}(\Sigma))$-invariant, it must contain $\Lambda_{\rho}$. Because $\rho_{l}$ and $\rho_{r}$ are holonomies of hyperbolic metrics on $\Sigma$ with geodesic boundary, the limit set $\Lambda_{\rho}$ is a Cantor set and we need to describe the remaining part of the boundary at infinity of the maximal surface.  This will be studied in the next subsections by comparing, in a neighbourhood of the punctures, the maximal embedding $\tilde{\sigma}$ and a particular maximal surface, called the horospherical surface (\cite{bon_schl}). 

\subsection{The frame field of a maximal embedding} Let us consider $\R^{4}\subset \C^{4}$ and extend the $\R$-bilinear form of signature $(2,2)$ to the Hermitian product on $\C^{4}$ given by
\[
	\langle z,w\rangle=z_{1}\bar{w}_{1}+z_{2}\bar{w}_{2}-z_{3}\bar{w}_{3}-z_{4}\bar{w}_{4} \ .
\]
Given a maximal conformal embedding $\tilde{\sigma}:\h^{2} \rightarrow \AdS_{3}$, with a slight abuse of notation, we still denote with $\tilde{\sigma}:\h^{2}\rightarrow \widehat{\AdS_{3}} \subset \C^{4}$ one of its lifts. Let $N$ be the unit normal vector field such that $\{\tilde{\sigma}_{w}, \tilde{\sigma}_{\bar{w}}, N, \tilde{\sigma}\}$ is an oriented frame in $\C^{4}$. We define
\[
	q=\langle N_{w}, \tilde{\sigma}_{\bar{w}} \rangle \ .
\]
The embedding being maximal implies that $q$ is a holomorphic quadratic differential on $\h^{2}$. Since the embedding is conformal, we can define a function $\phi:\h^{2} \rightarrow \R$ such that
\[
	\langle \tilde{\sigma}_{w}, \tilde{\sigma}_{w} \rangle= \langle \tilde{\sigma}_{\bar{w}}, \tilde{\sigma}_{\bar{w}} \rangle=e^{2\phi} \ .
\]
These are related to the embedding data of $\tilde{\sigma}$ as follows: the induced metric on $\tilde{\sigma}(\h^{2})$ is $I=2e^{2\phi}|dw|^{2}$ and the second fundamental form is $II=\Re(2q)$. The vectors 
\[
	v_{1}=\frac{\tilde{\sigma}_{w}}{e^{\phi}} \ \ \ v_{2}=\frac{\tilde{\sigma}_{\bar{w}}}{e^{\phi}} \ \ \ N, \ \ \ \text{and} \ \ \ \tilde{\sigma}
\]
give a unitary frame of $(\C^{4}, \langle \cdot, \cdot \rangle)$ at every point $w\in \h^{2}$. Taking the derivatives of the fundamental relations
\[
	\langle N,N\rangle=\langle \tilde{\sigma},\tilde{\sigma} \rangle=-1 \ \ \langle v_{j}, N \rangle=\langle v_{j},\tilde{\sigma}\rangle=0 \ \ \langle N_{z}, \tilde{\sigma}_{\bar{w}} \rangle=q \ \ \langle v_{j},v_{j}\rangle=1
\]
one deduces that
\[
	N_{\bar{w}}=e^{-\phi}\bar{q}v_{1} \ \ \ \overline{\partial}v_{1}=-\phi_{\bar{w}}v_{1}+e^{\phi}\tilde{\sigma} \ \ \ \text{and} \ \ \ \overline{\partial}v_{2}=\phi_{\bar{w}}v_{2}+\bar{q}e^{-\phi}N \ .
\]
Therefore, the pull-back of the Levi-Civita connection $\nabla$ of $(\C^{4}, \langle \cdot, \cdot, \rangle)$ via $\tilde{\sigma}$ can be written in the frame $\{v_{1},v_{2},N,\tilde{\sigma}\}$ as
\begin{equation}\label{eq:framefield}
	\tilde{\sigma}^{*}\nabla=Vd\bar{w}+Udw=
		\begin{pmatrix} -\phi_{\bar{w}} & 0 & e^{-\phi}\bar{q} & 0 \\
				0 & \phi_{\bar{w}} & 0 & e^{\phi} \\ 
				0 & e^{-\phi}\bar{q} & 0 & 0 \\
				e^{\phi} & 0 & 0 & 0
		\end{pmatrix}d\bar{w}+ \begin{pmatrix} \phi_{w} & 0 & 0 & e^{\phi} \\
							0 & -\phi_{w} & qe^{-\phi} & 0 \\
							qe^{-\phi} & 0 & 0 & 0 \\
							0 & e^{\phi} & 0 & 0 
				       \end{pmatrix}dw \ .
\end{equation}
Notice that the flatness of $\tilde{\sigma}^{*}\nabla$ is equivalent to $\phi$ being a solution of the PDE
\[
	\frac{1}{2}\Delta \phi= e^{2\phi}-e^{-2\phi}|q|^{2}
\]
which coincides with Equation (\ref{eq:PDE}) when the background metric is flat. \\

Viceversa, if a holomorphic quadratic differential $q$ and a solution $\phi$ of the above equation are given, the $1$-form $Vd\bar{w}+Udw$ can be integrated to a map $F:\h^{2}\rightarrow \SL(4,\C)$, which is the frame field of a maximal embedding into $\AdS_{3}$ with induced metric $I=2e^{2\phi}|dw|^{2}$ and second fundamental form $II=\Re(2q)$. Moreover, this is unique once the initial conditions are fixed.

\subsection{The horospherical surface}The frame field can be written explicitly in the special case when $q$ is a constant holomorphic quadratic differential, and the associated maximal surface in $\AdS_{3}$ appears in the literature as the horospherical surface (\cite{bon_schl}, \cite{seppimaximal}, \cite{TambuCMC}). See also \cite{Tambu_poly} and \cite{Tambu_regularAdS}.\\

Suppose $q=d\omega^{2}$ is a holomorphic quadratic differential defined on the complex plane $\C$. The corresponding solution to the flatness equation is then clearly $\phi=0$. The $1$-form becomes 
\[
	V_{0}d\bar{\omega}+U_{0}d\omega=\begin{pmatrix}
			0 & 0 & 1 & 0 \\
			0 & 0 & 0 & 1 \\
			0 & 1 & 0 & 0 \\
			1 & 0 & 0 & 0 
			\end{pmatrix} d\bar{\omega}+ \begin{pmatrix}
						0 & 0 & 0 & 1 \\
						0 & 0 & 1 & 0 \\
						1 & 0 & 0 & 0 \\
						0 & 1 & 0 & 0 
						\end{pmatrix}d\omega \ .
\]
The frame field of the horospherical surface is thus 
\[
	F_{0}(\omega)=A_{0}\exp(U_{0}\omega+V_{0}\bar{\omega}) \ ,
\]
for some constant matrix $A_{0} \in \SL(4, \C)$. For our convenience, we choose
\[
    A_{0}=\frac{1}{\sqrt{2}}\begin{pmatrix}
            1 & 1 & 0 & 0 \\
            -i & i & 0 & 0 \\
            0 & 0 & 1 & 1 \\
            0 & 0 & -1 & 1
            \end{pmatrix} \ .
\]
A simple computation shows that the matrix $U_{0}\omega+V_{0}\bar{\omega}$ is diagonalisable by a constant unitary matrix $R$ so that
\[
	R^{-1}(U_{0}\omega+V_{0}\bar{\omega})R=\diag(2\Re(\omega), 2\Imm(\omega), -2\Re(\omega), -2\Imm(\omega)) \ .
\]
Therefore, we can write 
\[
	F_{0}(\omega)=A_{0}R\diag(e^{2\Re(\omega)}, e^{2\Imm(\omega)}, e^{-2\Re(\omega)}, e^{-2\Imm(\omega)})R^{-1} \ . 
\]
The resulting maximal embedding is given by the last column of $F_{0}(\omega)$, that is
\[
    \sigma_{0}=\frac{1}{\sqrt{2}}(\sinh(2\Re(\omega)), \sinh(2\Imm(\omega)), \cosh(2\Re(\omega)), \cosh(2\Imm(\omega))) \ .
\]
In particular, we can describe explicitly the boundary at infinity $\Delta$ of $\sigma_{0}$: it consists of four light-like segments as the following table shows.

\medskip
\begin{table}[!htb]
\begin{center}
\begin{tabular}{l l l}
\hline
\textbf{Direction $\theta$} & \textbf{Projective limit of $\sigma_{0}(te^{i\theta}+iy)$}\\
\hline
$\theta \in (-\tfrac{\pi}{4}, \tfrac{\pi}{4})$ & $v_\theta = [1,0,1,0]$\\
$\theta = \tfrac{\pi}{4}$ & $v_y=[1,s,1,s]$ for some $s(y) \in \R^{+}$\\
$\theta \in (\tfrac{\pi}{4}, \tfrac{3\pi}{4})$ & $v_\theta = [0,1,0,1]$\\
$\theta = \tfrac{3\pi}{4}$ & $v_y=[-s,1,s,1]$ for some $s(y) \in \R^{+}$\\
$\theta\in (\tfrac{3\pi}{4}, \tfrac{5\pi}{4})$  & $v_\theta = [-1,0,1,0]$\\
$\theta = \tfrac{5\pi}{4}$ & $v_y=[-1,-s,1,s]$ for some $s(y) \in \R^{+}$\\
$\theta\in (\tfrac{5\pi}{4}, \tfrac{7\pi}{4})$  & $v_\theta = [0,-1,0,1]$\\
$\theta = \tfrac{7\pi}{4}$ & $v_y=[s,-1,s,1]$ for some $s(y) \in \R^{+}$\\
\hline\\
\end{tabular}
\end{center}
\caption{Limits of the standard horospherical surface along rays}\label{table:1}
\end{table}

\subsection{Comparison with the horospherical surface}\label{subsec:comparison} We saw in Section \ref{subsec:QD} that in a neighbourhood of a pole $p$ of order $n$ we can find $n-2$ standard half-planes $(U_{k}, \omega_{k})$ in which the quadratic differential $q$ pulls-back to $d\omega_{k}^{2}$. Moreover, the estimates in Proposition \ref{prop:existence}, show that very close to the puncture the induced metric on the maximal surface is approximated by the flat metric $|d\omega_{k}|^{2}$. This suggests that the equivariant maximal embedding $\tilde{\sigma}$ restricted to each half-plane should behave asymptotically as the horospherical surface. In order to make this idea precise, we adapt to this Lorentzian context the techiniques developed in \cite{DW}. 

Let $(U, \omega)$ be a standard half-plane. In this discussion we remove the dependence on the index $k$ with the understanding that the argument should be applied to each half-plane. Let $F:U \rightarrow \SL(4,\C)$ be the frame field of the maximal surface $\tilde{\sigma}$ found in Theorem \ref{teo:maxsurface} restricted to $U$. We define the osculating map $G:U\rightarrow \SO_{0}(2,2)$ by 
\[
	G(\omega)=F(\omega)F_{0}^{-1}(\omega)
\]
where $F_{0}:U \rightarrow \SL(4,\C)$ denotes the frame field of the horospherical surface. Notice that the map actually takes value in $\SO_{0}(2,2)$ because both frames $F(\omega)$ and $F_{0}(\omega)$ lie in the same right coset of $\SO_{0}(2,2)$ within $\SL(4,\C)$. Evidently, $G$ is constant if and only if $\tilde{\sigma}$ is itself a horospherical surface. A computation using the structure equation for a maximal surface shows that 
\[
	G^{-1}dG=F_{0}\Theta F_{0}^{-1} \ ,
\]
where 
\[
	\Theta(\omega)=\begin{pmatrix} 
				-v_{\bar{\omega}} & 0 & e^{-v}-1 & 0 \\
				0 & v_{\bar{\omega}} & 0 & e^{v}-1 \\
				0 & e^{-v}-1 & 0 & 0 \\
				e^{v}-1 & 0 & 0 & 0 
				\end{pmatrix}d\bar{\omega} \ + 
\]
\[
				 \ \ \ \ \ \ \ \begin{pmatrix}
				v_{\omega} & 0 & 0 & e^{v}-1 \\
				0 & -v_{\omega} & e^{-v}-1 & 0 \\
				e^{-v}-1 & 0 & 0 & 0 \\
				0 & e^{v}-1 & 0 & 0 
				\end{pmatrix}d\omega
\]
and $I=2e^{2v}|d\omega|^{2}$ is the induced metric on the maximal surface in the $\omega$-coordinate. Notice that the estimates in Proposition \ref{prop:estimates} show that $\Theta(\omega)$ is rapidly decaying to $0$ as $|\omega|$ increases. Ignoring the conjugation by the matrix $F_{0}(\omega)$, this suggests that $G(\omega)$ should converge to a constant as $\omega$ goes to infinity, which would mean that the maximal surface $\tilde{\sigma}(\tilde{\Sigma})$ is asymptotic to a horospherical surface. However, the frame field $F_{0}$ is itself exponentially growing, with a precise rate depending on the direction. Thus the actual asymptotic behaviour of $G$ depends on the comparison between the growth of the error $\Theta(\omega)$ and the frame field $F_{0}(\omega)$. In most directions, the exponential decay of $\Theta(\omega)$ is faster than the growth of $F_{0}(\omega)$, giving a well-defined limiting horospherical surface. In exactly $2$ directions there is an exact balance, which allow the horospherical surface to shift. We start by pointing out the stable directions:

\begin{defi} We say that a ray $\gamma(t)=e^{i\theta}t+iy$ is stable if $\theta \notin \{\pi/4, 3\pi/4\}$.  
\end{defi}

Notice that the possible directions of stable rays in a standard half-plane form three open intervals
\[
J_{+}=(0, \pi/4) \ \ \ \ J_{0}=(\pi/4, 3\pi/4) \ \ \ \ \text{and} \ \ \ \ J_{-}=(3\pi/4, \pi) .
\]
\begin{lemma}\label{lm:stable} If $\gamma$ is a stable ray, then $\lim_{t \to +\infty}G(\gamma(t))$ exists. Furthermore, among all such rays only three limits are achieved: there exist $L_{0}, L_{\pm} \in \SO_{0}(2,2)$ such that
\[
    \lim_{t \to +\infty}G(\gamma(t))=\begin{cases}
                    L_{+} \ \ \ \ \ \ \ \text{if} \ \ \theta \in J_{+} \\
                    L_{0} \ \ \ \ \ \  \ \text{if} \ \ \theta \in J_{0} \\
                    L_{-} \ \ \ \ \ \ \text{if} \ \ \theta \in J_{-}
                    \end{cases}
\]
\end{lemma}
\begin{proof}Let $\gamma(t)=e^{i\theta}t+iy$ be a ray in a stable direction $\theta$. For brevity we denote $G(t)=G(\gamma(t))$. We know that
\[
    G(t)^{-1}G'(t)=F_{0}(\gamma(t))\Theta_{\gamma(t)}(\dot{\gamma}(t))F_{0}^{-1}(\gamma(t)) \ .
\]
Since $F_{0}(w)=A_{0}R\diag(e^{2\Re(w)}, e^{2\Im(w)}, e^{-2\Re(w)}, e^{-2\Im(w)})R^{-1}$, for constant matrices $R$ and $A_{0}$, the asymptotic behaviour of $F_{0}(\gamma(t))\Theta_{\gamma(t)}(\dot{\gamma}(t))F_{0}^{-1}(\gamma(t))$ depends only on the action by conjugation by the diagonal matrix
\[
    D(t)=\diag(e^{2\Re(\gamma(t))}, e^{2\Im(\gamma(t))}, e^{-2\Re(\gamma(t))}, e^{-2\Im(\gamma(t))}) \ .
\]
A direct computation shows that $R^{-1}\Theta R$ is equal to 
\[ 
    \frac{e^{-i\theta}}{4}\scalebox{0.75}{$\begin{pmatrix}
    2(e^{v}+e^{-v}-2) & -2v_{\bar{\omega}}+(1-i)(e^v-e^{-v}) & 0 & -2v_{\bar{\omega}}+(1+i)(e^v-e^{-v}) \\
    -2v_{\bar{\omega}}-(1-i)(e^v-e^{-v}) & 2i(e^v+e^{-v}-2) & -2v_{\bar{\omega}}+(1+i)(e^v-e^{-v}) & 0 \\
    0 & -2v_{\bar{\omega}}-(1+i)(e^v-e^{-v}) & -2(e^{v}+e^{-v}-2) & -2v_{\bar{\omega}}-(1-i)(e^v-e^{-v})\\
    -2v_{\bar{\omega}}-(1+i)(e^v-e^{-v}) & 0 & -2v_{\bar{\omega}}+(1-i)(e^v-e^{-v}) & -2i(e^{v}+e^{-v}-2)
    \end{pmatrix}dt$} + 
\]
\[
     \frac{e^{i\theta}}{4}\scalebox{0.75}{$\begin{pmatrix}
    2(e^{v}+e^{-v}-2) & 2v_{\omega}-(1+i)(e^v-e^{-v}) & 0 & 2v_{\omega}-(1-i)(e^v-e^{-v}) \\
    2v_{\omega}+(1+i)(e^v-e^{-v}) & 2i(e^v+e^{-v}-2) & 2v_{\omega}-(1-i)(e^v-e^{-v}) & 0 \\
    0 & 2v_{\omega}+(1-i)(e^v-e^{-v}) & -2(e^{v}+e^{-v}-2) & 2v_{\omega}+(1+i)(e^v-e^{-v})\\
    2v_{\omega}+(1-i)(e^v-e^{-v}) & 0 & 2v_{\omega}-(1+i)(e^v-e^{-v}) & -2i(e^{v}+e^{-v}-2)
    \end{pmatrix}dt$}
\]
and conjugating by $D(t)$ multiplies the $(i,j)$-entry by
\[
    \lambda_{ij}=\exp\left(2t \left(\cos\left(\theta+\frac{(i-1)\pi}{2}\right)+\cos\left(\theta+\frac{(j-1)\pi}{2}\right)\right)\right)=O\left(e^{c(\theta)t}\right) \ , 
\]
where $c(\theta)$ achieves its maximum $2\sqrt{2}$ at $\theta=\pm\pi/4$ (here we have considered only the pairs $(i,j)$ so that $\lambda_{ij}$ multiplies a non-zero entry). Combining the bounds for $R^{-1}\Theta R$ and $\lambda_{ij}$, we find that for every stable ray, 
\[
    G(t)^{-1}G'(t)=O\left(\frac{e^{-\beta t}}{\sqrt{t}}\right)
\]
where $\beta=2\sqrt{2}-c(\theta)>0$. It is then standard to show that the limit $\lim_{t \to +\infty}G(t)$ exists (\cite[Lemma B.1]{DW}). \\
\indent Now suppose that $\gamma_{1}$ and $\gamma_{2}$ are stable rays with respective angles $\theta_{1}$ and $\theta_{2}$ that belong to the same interval. For any $t\geq 0$, let $\eta_{t}(s)=(1-s)\gamma_{1}(t)+s\gamma_{2}(t)$ be the constant-speed parameterisation of the segment from $\gamma_{1}(t)$ to $\gamma_{2}(t)$. Let $g_{t}(s)=G(\eta_{t}(0))^{-1}G(\eta_{t}(s))$, which satisfies
\[
    \begin{cases} g_{t}^{-1}(s)g'_{t}(s)=F_{0}(\eta_{t}(s))\Theta_{\eta_{t}(s)}(\dot{\eta}_{t}(s))F_{0}^{-1}(\eta_{t}(s)) \\
                    g_{t}(0)=Id \\
                    g_{t}(1)=G_{1}(t)^{-1}G_{2}(t)
    \end{cases} \ ,
\]
where $G_{i}(t)=G(\gamma_{i}(t))$ for $i=1,2$. Since $|\dot{\eta}_{t}(s)|=O(t)$, the analysis above shows that 
\[
    g_{t}^{-1}(s)g_{t}'(s)=O\left(\sqrt{t}e^{-\beta t}\right) \ ,
\] 
where $\beta=2\sqrt{2}-\sup\{ c(\theta) \ | \ \theta_{1} \leq \theta \leq \theta_{2} \}$. In particular, by making $t$ large we can arrange $g_{t}^{-1}(s)g_{t}'(s)$ to be uniformly small for all $s \in [0,1]$. ODE methods (\cite[Lemma B.1]{DW}) ensure that
\[
    g_{t}(1)=G_{1}^{-1}(t)G_{2}(t) \to Id \ \ \ \text{as} \ \ t \to +\infty \ .
\]
This shows that $G$ has the same limit along $\gamma_{1}$ and $\gamma_{2}$. \\
\end{proof}

Next we analyse the behaviour across unstable rays in order to understand the relationship between $L_{\pm}$ and $L_{0}$.

\begin{lemma}\label{lm:unipotent} Let $L_{\pm}$ and $L_{0}$ be as in the previous lemma. Then there exist unipotent matrices $U_{\pm}$ such that
\[
    L_{+}^{-1}L_{0}=A_{0}RU_{+}R^{-1}A_{0}^{-1} \ \ \ \ \ \text{and} \ \ \ \ \ L_{0}^{-1}L_{-}=A_{0}RU_{-}R^{-1}A_{0}^{-1} \ .
\]
\end{lemma}
\begin{proof} We give the detailed proof for $L_{+}^{-1}L_{0}$, for the other case we only underline the differences at the end. Consider the rays $\gamma_{+}(t)=e^{i\pi/6}t$ and $\gamma_{0}(t)=it$. By the previous lemma $G_{+}(t)=G(\gamma_{+}(t))$ and $G_{0}(t)=G(\gamma_{0}(t))$ have limit $L_{+}$ and $L_{0}$, respectively. For any $t>0$, we join $\gamma_{+}(t)$ and $\gamma_{0}(t)$ by an arc
\[
    \eta_{t}(s)=e^{is}t \ , \ \ \ \text{where} \ s \in [\pi/6, \pi/2] \ .
\]
Let $g_{t}(s)=G(\eta_{t}(\pi/6))^{-1}G(\eta_{t}(s))$. Then $g_{t}:[\pi/6, \pi/2]\rightarrow \SO_{0}(2,2)$ satisfies the differential equation
\begin{equation}\label{eq:ODE}
    \begin{cases} g_{t}^{-1}(s)g'_{t}(s)=F_{0}(\eta_{t}(s))\Theta_{\eta_{t}(s)}(\dot{\eta}_{t}(s))F_{0}^{-1}(\eta_{t}(s)) \\
                    g_{t}(\pi/6)=Id \\
                    g_{t}(0)=G_{+}(t)^{-1}G_{0}(t)
    \end{cases} \ .
\end{equation}  
Unlike the previous case, the coefficient  
\[
    M_{t}(s)=D(\eta_{t}(s))R^{-1}\Theta_{\eta_{t}(s)}(\dot{\eta}_{t}(s))RD(\eta_{t}(s))^{-1} 
\]
is not exponentially small in $t$ throughout the interval. At $s=\pi/4$, conjugation by $D(\eta_{t}(\pi/4))$ multiplies the $(1,4)$-entry and the $(2,3)$-entry by a factor $\exp(2\sqrt{2}t)$, exactly matching the decay rate of $R^{-1}\Theta R$ and giving
\[
    M_{t}(\pi/4)=O\left(\frac{|\dot{\eta}_{t}(\pi/4)|}{\sqrt{t}}\right)=O(\sqrt{t})
\]
because $|\dot{\eta}_{t}(\pi/4)|=t$. However, this growth is seen only in the $(1,4)$-entry and in the $(2,3)$-entry because all the others are scaled by a smaller exponential factor. Moreover, for $\theta \in [\pi/6, \pi/2]$ we have
\[
    \lambda_{14}=\lambda_{23}=\exp(2t(\cos\theta+\sin\theta))\leq \exp\left(2\sqrt{2}t-\left(\theta-\frac{\pi}{4}\right)^{2}t\right) \ ,
\]
thus we can separate the unbounded entry in $M_{t}(s)$ and write
\[
    M_{t}(s)=M_{t}^{0}(s)+\mu_{t}(s)(E_{14}+E_{23})
\]
where $M_{t}^{0}(s)=O(e^{-\beta t})$ for some $\beta>0$, $E_{14}$ and $E_{23}$ are the elementary matrices, and
\[
    \mu_{t}(s)=O\left(|\dot{\eta}_{t}(s)|\exp(-2\sqrt{2}t)\lambda_{14}\right)=O\left(\sqrt{t}e^{-(\theta-\pi/4)^{2}t}\right) \ .
\]
This upper-bound is a Gaussian function centered at $\theta=\pi/4$, normalised such that its integral is independent of $t$. Therefore, the function $\mu_{t}(s)$ is uniformly absolutely integrable over $s \in [\pi/6, \pi/2]$ as $t \to +\infty$. Now, under this condition the solution to the initial value problem (\ref{eq:ODE}) satisfies (\cite[Lemma B.2]{DW})
\[
    \left\| g_{t}(\pi/2)-A_{0}R\exp\left((E_{14}+E_{23})\int_{\frac{\pi}{6}}^{\frac{\pi}{2}}\mu_{t}(s)\right)R^{-1}A_{0}^{-1}\right\| \to 0 \ \ \text{as} \ \ t \to +\infty \ .
\]
Since $g_{t}(0)=G(\gamma_{+}(t))^{-1}G(\gamma_{0}(t)) \to L_{+}^{-1}L_{0}$, this gives the desired unipotent form. \\
The proof for $L_{0}^{-1}L_{-}$ follows the same line with the only difference given by the fact that at $\theta=3\pi/4$, the leading term in the matrix $M_{t}(s)$ lies in the $(2,1)$- and $(3,4)$-entry.
\end{proof}

\subsection{Light-like polygonal ends} We are now going to use the above estimates in order to describe the boundary at infinity $\Gamma$ of the $\rho$-equivariant maximal surface $\tilde{\sigma}(\tilde{\Sigma})$. We already know that $\Gamma$ contains the limit set of the representation $\rho=(\rho_{l}, \rho_{r})$ which is a Cantor set consisting of the pairs of attracting fixed points of the holonomy. On the other hand, $\Gamma$ must be a locally achronal topological circle. By equivariance, what remains to be understood is how the points $x^{++}(\rho(\gamma))$ and $x^{--}(\rho(\gamma))$ are connected, for every peripheral element $\gamma \in \pi_{1}(\Sigma)$. We will prove the following.

\begin{teo} \label{teo:ends} Let $p$ be a puncture and suppose that the quadratic differential $q$ has a pole of order $n\geq 3$ at $p$. Let $\gamma \in \pi_{1}(\Sigma)$ be a peripheral element around $p$. Then the points $x^{++}(\rho(\gamma))$ and $x^{--}(\rho(\gamma))$ are connected by a $\rho(\gamma)$-equivariant achronal light-like polygonal end with $2(n-2)$ fundamental vertices and accumulation points $x^{++}(\rho(\gamma))$ and $x^{--}(\rho(\gamma))$. 
\end{teo}

Let us first explain the terminology.  A light-like polygonal end $\mathcal{LP}$ in the boundary at infinity of $\AdS_{3}$ is a concatenation of light-like segments $\{ \ell_{i}\}_{i \in \Z}$. This can be constructed by alternating segments belonging to the right- and left- foliation of $\partial_{\infty}\AdS_{3}$. We say that the polygonal end is $\rho(\gamma)$-equivariant if there exists an integer $k>0$ such that $\rho(\gamma)\ell_{i}=\ell_{i+k}$ for every $i \in \Z$. In this case, we can reconstruct the polygonal end by knowing a finite number of vertices, that we called fundamental, as they can be obtained by considering the vertices contained in a fundamental domain of the action of $\rho(\gamma)$ in $\partial_{\infty}\AdS_{3}$. The accumulation points are the limits
\[
	\mathrm{Accum}(\mathcal{LP})=\lim_{i \to +\infty} \ell_{i} \cup \lim_{i \to -\infty} \ell_{i} \ .
\]
Clearly, if $\mathcal{LP}$ is equivariant with respect to the action of $\rho(\gamma)$, the accumulation points always coincide with the attracting and repelling fixed points of $\rho(\gamma)$. \\

The vertices of the polygonal end will arise as limits of the equivariant maximal embedding $\tilde{\sigma}$ along lifts of paths on $\Sigma$ directed towards the puncture. The limit will depend on the homotopy class of the paths and we will pin down special representatives in each homotopy class in order to be able to apply the estimates in Section \ref{subsec:comparison}. We are going to consider paths $\beta$ starting from a fixed base point in $\Sigma$ and converging to the puncture $p$ which in a natural coordinate are rays. This allows us to talk about the direction $\theta$ in which $\beta$ is approaching the puncture. Notice that paths in different homotopy classes can converge to the puncture along the same direction in the same half-plane, as they may differ by a complete rotation along the puncture. A way to construct such paths is the following: let $\{(U_{k}, \omega_{k})\}_{k=1}^{n-2}$ be the collection of standard half-planes that cyclically cover a neighbourhood of the puncture. A path $\beta_{k,\theta}$ converging to the puncture in the half-plane $U_{k}$ with direction $\theta$ can be obtained by concatenating a path $\alpha_{k}:[0,1]\rightarrow \Sigma$ that connects the fixed base point on $\Sigma$ and the boundary of $U_{k}$ with the ray $\gamma_{\theta}(t)=e^{i\theta}t+\alpha_{k}(1)$. Morover, we can choose the paths $\alpha_{k}$ so that $\alpha_{1}$ is contractible and $\alpha_{k}$ is obtained from $\alpha_{1}$ by following the boundary of the halfplanes (see Figure \ref{fig:paths}). 

\begin{figure}[h!]
\begin{center}
\includegraphics[scale=.3]{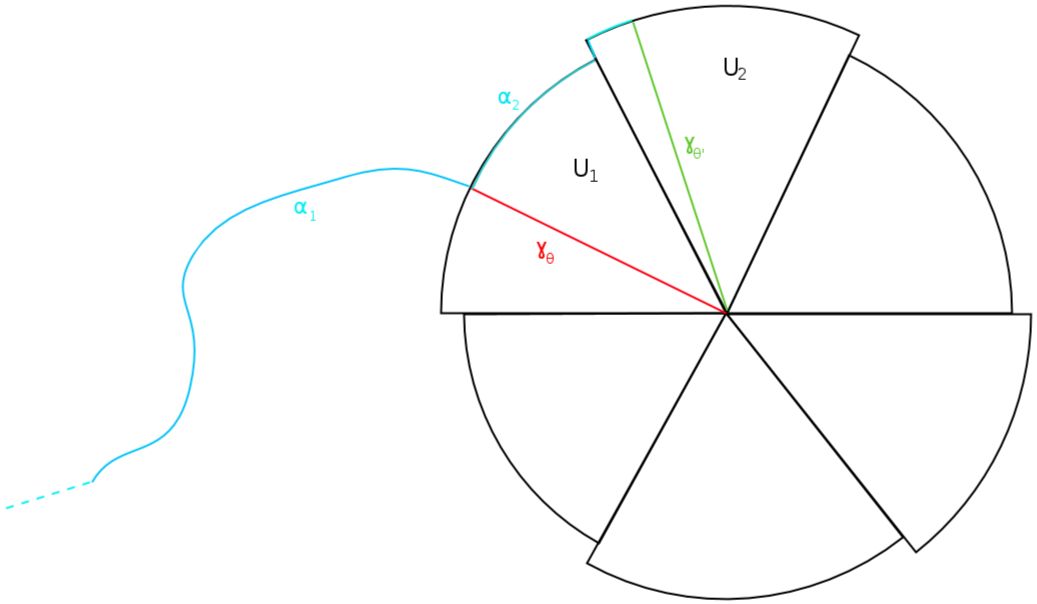}
\caption{Definition of the paths $\beta_{k,\theta,m}$}\label{fig:paths}
\end{center}
\end{figure}

In this way, when $m \in \Z$ rotations are completed around the puncture we get paths in different homotopy classes approaching the puncture in the same direction $\theta$ and in the same half-plane $U_{k}$ as $\beta_{k, \theta}$. We will denote such paths by $\beta_{k, \theta, m}$ with the convention that $\beta_{k, \theta, 0}=\beta_{k,\theta}$. Each homotopy class of paths converging to the puncture has then a representative of the form $\beta_{k,\theta,m}\cdot \alpha$ for some $\alpha \in \pi_{1}(\overline{\Sigma})$.

\begin{proof}[Proof of Theorem \ref{teo:ends}] Let us consider the family of paths $\beta_{k, \theta, m}$ converging to the puncture $p$ as above. By Lemma \ref{lm:stable}, the limit $L^{m}_{k, \theta}=\lim_{t \to + \infty}G(\beta_{k,\theta,m}(t))$ exists as long as $\theta$ is a stable direction, and only depends on the interval $J_{\pm}$ or $J_{0}$ in which $\theta$ lies. We will thus denote the limiting matrix as $L^{m}_{k, \epsilon}$ where $\epsilon=0,\pm$, keeping track only on such interval. Since the frame field of the maximal embedding is $F(\beta_{k,\theta,m}(t))=G(\beta_{k,\theta,m}(t))F_{0}(\beta_{k, \theta, m}(t))$, the limiting point along the path can be expressed as $\nu^{m}_{k, \epsilon}=L^{m}_{k, \epsilon}v_{\theta}$, where $v_{\theta}$ is the point at infinity of the standard horospherical surface along a ray with direction $\theta$ (see Table \ref{table:1}). We deduce that in each standard half-plane $U_{k}$, for a fixed $m \in \Z$, we see three points at infinity as the following table shows:

\medskip
\begin{table}[!htb]
\begin{center}
\begin{tabular}{l l l}
\hline
\textbf{Direction $\theta$} & \textbf{Projective limit of $\ell^{m}_{k, \epsilon}$ of $\tilde{\sigma}$}\\
\hline
$\theta \in (0, \tfrac{\pi}{4})$ & $\nu^{m}_{k, +} = L^{m}_{k,+}[1,0,1,0]$\\
$\theta \in (\tfrac{\pi}{4}, \tfrac{3\pi}{4})$ & $\nu^{m}_{k, 0} = L^{m}_{k,0}[0,1,0,1]$\\
$\theta\in (\tfrac{3\pi}{4}, \pi)$  & $\nu^{m}_{k, -} = L^{m}_{k,-}[-1,0,1,0]$\\
\hline\\
\end{tabular}
\end{center}
\caption{Limits along rays in a half-plane}\label{table: 2}
\end{table}

\noindent A direct computation, using the formulas provided by Lemma \ref{lm:unipotent}, shows that
\[
	L^{m}_{k,0}[-1,0,1,0]=L^{m}_{k,-}[-1,0,1,0] \ \ \ \text{and} \ \ \ L^{m}_{k,0}[1,0,1,0]=L^{m}_{k,+}[1,0,1,0] \ .
\]
In particular, the three limit points that appear in each half-plane are causally related, being them the image under an element of $\SO_{0}(2,2)$ of causally related points, and the light-like segment joining them is entirely contained in the boundary at infinity. Therefore, for $m=0$ in each standard hall-plane $U_{k}$, we see a "vee" in the boundary at infinity of the maximal surface given by 
\[
	L^{0}_{k,0}([1,0,1,0]\cup [1,s,1,s] \cup [0,1,0,1] \cup [-s, 1, s, 1] \cup [-1,0,1,0])
\]
with $s$ varying in $\R^{+}$. Given two consecutive half-planes $U_{k}$ and $U_{k+1}$ the two "vees" share an extreme vertex: in fact, by considering another standard half-plane $W$ that intersects $U_{k}$ and $U_{k+1}$ in a sector of angle $\pi/2$, the arguments in Lemma \ref{lm:stable} show that the direction $\pi$ is stable, so the ending point of the "vee" in $U_{k}$ coincides with the first point of the "vee" in $U_{k+1}$. Notice that when $k=n-2$ this procedure makes the index $m$ of the path increase by one. The above discussion thus produces a collection of vertices $\{\nu_{i}^{m}\}_{i=1}^{2(n-2)}$ in the boundary at infinity of the maximal surface that arise as limits along the paths $\beta_{k,\theta,m}$ with $\nu^{m}_{2(n-2)}=\nu^{m+1}_{1}$, such that two consecutive vertices are connected by light-like segments that belong to the left- and right-foliation alternately. Moreover, since for every $m\in \Z$ the path $\beta_{k,\theta,m}$ is homotopic to $\beta_{k,\theta, 0}\cdot \gamma^{m}$, where $\gamma$ is the peripheral element that goes once around the puncture, we have that, if
\[
	\nu_{i}^{0}=\lim_{t \to \infty}\tilde{\sigma}(\beta_{k,\theta,0}(t))
\]
then
\begin{align*}
	\nu_{i}^{m}=\lim_{t \to \infty}\tilde{\sigma}(\beta_{k,\theta,m}(t))=\lim_{t \to \infty}\tilde{\sigma}((\beta_{k, \theta, 0}\cdot \gamma^{m})(t))=\rho(\gamma)^{m}\nu_{i}^{0}
\end{align*}
hence the polygonal end is $\rho(\gamma)$-equivariant and there are $2(n-2)$ fundamental vertices. 
\end{proof}

\section{Parameterisation of wild anti-de Sitter structures}\label{sec:para}
From the results of the previous sections, we can construct a globally hyperbolic anti-de Sitter structure from the data of a complete hyperbolic metric $h$ of finite area on $\Sigma$ and a meromorphic quadratic differential $q$ with poles of order at least $3$ at the punctures. Namely, Theorem \ref{teo:maxsurface} provides a unique complete equivariant maximal embedding into $\AdS_{3}$ whose boundary at infinity is an achronal curve $\Gamma(h,q)$ that contains the limit set of the holonomy completed to a topological circle by inserting light-like polygonal ends for each peripheral element. Let $\Omega(h,q)$ be the domain of dependence of this boundary curve. The holonomy representation acts properly discontinuously on $\Omega(h,q)$ (\cite{MR2369412},\cite{MR2443264}) and the quotient is a globally hyperbolic maximal anti-de Sitter manifold $M(h,q)$ diffeomorphic to $\Sigma\times\R$. On the other hand, for a fixed pair of discrete and faithful representations $\rho_{l,r}:\pi_{1}(\Sigma)\rightarrow \PSL(2,\R)$ with hyperbolic peripheral elements, the space of GHM anti-de Sitter structures $\mathcal{GH}(\Sigma)$ on $\Sigma\times \R$ is quite large: if $\Lambda_{\rho}$ is the limit set of the action of $\rho=(\rho_{l},\rho_{r})$, then there is a one-to-one correspondence between elements of $\mathcal{GH}(\Sigma)$ and $\rho(\pi_{1}(\Sigma))$-equivariant completions of $\Lambda_{\rho}$ to an achronal topological circle (\cite{bsk_multiblack}). The aim of this section is thus to characterise the image of the map 
\begin{align*}
	\Psi: \mathcal{MQ}_{\geq 3}(\Sigma) &\rightarrow \mathcal{GH}(\Sigma) \\
				(h,q) &\mapsto M(h,q)
\end{align*}
associating to a point $(h,q) \in \mathcal{MQ}_{\geq 3}(\Sigma)$ in the bundle over Teichm\"uller space of meromorphic quadratic differentials with poles of order at least $3$ at the punctures the corresponding GHM anti-de Sitter structure. 

\begin{prop}\label{prop:injective} The map $\Psi$ is injective.
\end{prop}
\begin{proof} Suppose by contradiction that $\Psi$ is not injective. Then we can find $(h,q)\neq (h',q') \in \mathcal{MQ}_{\geq 3}(\Sigma)$ such that $\Psi(h,q)=\Psi(h',q')$. By definition, this means that the equivariant maximal embeddings associated to $(h,q)$ and $(h',q')$ have the same holonomy representation and the same boundary at infinity. On the other hand, the arguments of \cite[Lemma 4.2]{Tambu_poly} show that given an achronal curve $\Gamma \subset \partial_{\infty}\AdS_{3}$ the maximal surface bounding $\Gamma$ is unique. This gives a contradiction because the pair $(h,q)$ is uniquely determined by the embedding data of the maximal surface.
\end{proof}

\begin{prop} The map $\Psi$ is continuous.
\end{prop}
\begin{proof}We endow $\mathcal{GH}(\Sigma)$ with the topology induced by the one-to-one correspondence between elements of $\mathcal{GH}(\Sigma)$ and pairs $(\rho, \Gamma_{\rho})$, where $\rho=(\rho_{l}, \rho_{r})$ is pair of discrete and faithful representations into $\PSL(2,\R)$ with hyperbolic peripheral elements and $\Gamma_{\rho}$ is an achronal completion of the limit set of $\rho$ to a topological circle. We thus consider on $\mathcal{GH}(\Sigma)$ the topology induced by the product of the usual topology in the space of representations and the Hausdorff topology for compact sets in $\partial_{\infty}\AdS_{3}$. \\
Let $(h_{n},q_{n})\in \mathcal{MQ}_{\geq 3}(\Sigma)$ be a sequence converging to $(h,q)\in\mathcal{MQ}_{\geq 3}(\Sigma)$. We need to prove that the holonomy representation of $M(h_{n},q_{n})$ converges to the holonomy representation of $M(h,q)$ and the boundary curves $\Gamma(h_{n},q_{n})$ converge to $\Gamma(h,q)$ in the Hausdorff topology. Let $v_{n}$ and $v$ be the solution to Equation (\ref{eq:PDE}) associated to the data $(h_{n},q_{n})$ and $(h,q)$ respectively. On every compact set $K\subset \Sigma$, the supersolution and subsolution found in Proposition \ref{prop:existence} provide a uniform bound for $\Delta_{h_{n}}v_{n}$. Since $h_{n}$ is a convergent sequence, standard theory for elliptic PDE gives a uniform $W^{1,2}$ bound for $v_{n}$. Thus $v_{n}$ subconverges to a weak solution of the equation 
\[
	\frac{1}{2}\Delta_{v}h=e^{2v}-e^{-2v}\|q\|^{2}_{h}+\frac{1}{2}K_{h} \ ,
\]
in $W^{1,2}$ on every compact set. By elliptic regularity $v$ is smooth and the convergence is actually smooth. We deduce that the embedding data of the unique maximal surface in $M(h_{n},q_{n})$ converges smoothly on compact sets to the embedding data of the unique maximal surface in $M(h,q)$. By lifting to the universal cover, this implies that the corresponding equivariant maximal embeddings $\tilde{\sigma_{n}}: \tilde{\Sigma}\rightarrow \AdS_{3}$ are converging smoothly on compact sets (up to post-composition by a global isometry) to $\tilde{\sigma}:\tilde{\Sigma}\rightarrow \AdS_{3}$, and thus the boundary at infinity $\Gamma(h_{n},q_{n})$ converges to $\Gamma(h,q)$ in the Hausdorff topology. The convergence of the holonomy follows from the general fact below, which was proved in \cite[Lemma 5.3]{Tambu_regularAdS}.
\end{proof}

\begin{lemma} Let $\tilde{\sigma}:\tilde{\Sigma}\rightarrow \AdS_{3}$ be a sequence of $\rho_{n}$-equivariant space-like embeddings. If $\tilde{\sigma}_{n}$ converges to a space-like embedding $\tilde{\sigma}$ smoothly on compact sets, then $\rho_{n}$ converges, up to subsequences, to a representation $\rho$ and $\tilde{\sigma}$ is $\rho$-equivariant.
\end{lemma}

We define the subset of \emph{wild} GHM anti-de Sitter structures on $\Sigma\times \R$ as the image of the map $\Psi$:
\[
	\mathcal{GH}^{wild}(\Sigma)=\Psi(\mathcal{MQ}_{\geq 3}(\Sigma))\subset \mathcal{GH}(\Sigma) \ .
\]

From Section \ref{sec:boundary} we know that the curves at infinity $\Gamma(h,q)$ are always obtained by completing the limit set of the holonomy with light-like polygonal ends, but at the moment we do not know if any representation $\rho=(\rho_{l}, \rho_{r})$ is attained and if any light-like polygonal end can be realised. Let us denote with $\mathcal{GH}^{\mathcal{LP}}(\Sigma)$ the space of globally hyperbolic anti-de Sitter manifolds (up to diffeomorphisms isotopic to the identity) with holonomy $\rho$ given by a pair of faithful and discrete representations with hyperbolic peripheral elements and boundary at infinity given by a light-like polygonal completion of the limit set of $\rho$. Hence $\mathcal{GH}^{wild}(\Sigma)\subset \mathcal{GH}^{\mathcal{LP}}(\Sigma)$.

\begin{prop}The map $\Psi: \mathcal{MQ}_{\geq 3}(\Sigma) \rightarrow \mathcal{GH}^{\mathcal{LP}}(\Sigma)$ is proper.
\end{prop}
\begin{proof} Let $(\rho_{n},\Gamma_{n})=\Psi(h_{n},q_{n})$ be a sequence of globally hyperbolic maximal anti-de Sitter structures with holonomy $\rho_{n}$ and boundary at infinity $\Gamma_{n}$ that lie in the image of the map $\Psi$. Suppose that $\rho_{n}$ and $\Gamma_{n}$ converge to $\rho$ and $\Gamma$ in $\mathcal{GH}^{\mathcal{LP}}(\Sigma)$. In particular, the boundaries at infinity $\Gamma_{n}$ of the maximal surfaces $S_{n}$ embedded in these manifolds converge to $\Gamma$ in the Hausdorff topology. The arguments of \cite[Section 4.1]{Tambu_poly} show that the sequence $S_{n}$ converges to a maximal surface $S$ bounding $\Gamma$ smoothly on compact sets. In particular, the embedding data of $S_{n}$ converge to the embedding data of $S$. We deduce that $h_{n}$ converges to $h$ in $\mathcal{T}(\Sigma)$ and $q_{n}$ converges to a meromorphic quadratic differential $q$. Moreover, $q$ has poles of order at least $3$ because the order of the poles determines the number of fundamental vertices in the light-like polygonal completion of the boundary at infinity of the maximal surface.
\end{proof}

Therefore, the map $\Psi$ is a homeomorphism onto its image, and we are going to prove that $\mathcal{GH}^{wild}(\Sigma)=\mathcal{GH}^{\mathcal{LP}}(\Sigma)$, by showing that for every $n_{1}, \dots, n_{N}\geq 3$ the subset $\mathcal{GH}^{\mathcal{LP}}(\Sigma, 2n_{1}-4, \dots, 2n_{N}-4) \subset \mathcal{GH}^{\mathcal{LP}}(\Sigma)$ consisting of light-like polygonal completions with $2n_{i}-4$ fundamental vertices is a manifold of the same dimension as the subbundle $\mathcal{MQ}(\Sigma, n_{1}, \dots, n_{N})$ of meromorphic quadratic differentials with poles of order exactly $n_{i}$. 

\subsection{Parameterisation of $\mathcal{GH}^{\mathcal{LP}}(\Sigma, 2n_{1}-4, \dots, 2n_{N}-4)$} A curve $\Gamma$ on the boundary at infinity of $\AdS_{3}$ can be seen as a graph of a function $f_{\Gamma}:\R\Pp^{1}\rightarrow \R\Pp^{1}$ in the following way. Fix a totally geodesic plane $P_{0}$ in $\AdS_{3}$. Its boundary at infinity is a circle in $\partial_{\infty}\AdS_{3}$. Any point $\xi\in \partial_{\infty}\AdS_{3}$ lies in a unique line belonging to the right-foliation and a unique line belonging to the left-foliation of $\partial_{\infty}\AdS_{3}$. Those two lines intersect the circle at infinity of $P_{0}$ in exactly one point that we denote by $\Pi_{l}(\xi)$ and $\Pi_{r}(\xi)$ respectively. We can thus associate to a curve $\Gamma$ the map $f_{\Gamma}:\R\Pp^{1}\rightarrow \R\Pp^{1}$ defined by the property that
\[
	f_{\Gamma}(\Pi_{l}(\xi))=\Pi_{r}(\xi) \ \ \ \ \text{for every $\xi \in \Gamma$.}
\]
This procedure gives a well-defined map, as soon as $\Gamma$ is an acausal curve (\cite{bon_schl}). However, if $\Gamma$ is a light-like polygonal completion of the limit set $\Lambda_{\rho}$ of a representation $\rho=(\rho_{l},\rho_{r}):\pi_{1}(\Sigma)\rightarrow \dPSL$, we can make this construction work and associate a unique function $f_{\Gamma}:\R\Pp^{1}\rightarrow \R\Pp^{1}$ as follows. By definition, the projections $\Lambda_{l}=\Pi_{l}(\Lambda_{\rho})$ and $\Lambda_{r}=\Pi_{r}(\Lambda_{\rho})$ are the limit sets of the representations $\rho_{l}$ and $\rho_{r}$ acting on $P_{0}$, which is isometric to the hyperbolic plane. We can set $f_{\Gamma}:\Lambda_{l}\rightarrow \Lambda_{r}$ as the unique map such that 
\[
	f_{\Gamma}(\rho_{l}(\gamma))=\rho_{r}(\gamma)\circ f_{\Gamma} \ \ \ \ \ \text{for every $\gamma \in \Gamma$.}
\]
Notice that, in particular, $f_{\Gamma}$ sends the attracting (resp. repelling) fixed points of $\rho_{l}$ to the attracting (resp. repelling) fixed points of $\rho_{r}$. We then want to extend this function to the whole $\R\Pp^{1}$. Consider a connected component $C$ of $\R\Pp^{1}\setminus \Lambda_{l}$: this corresponds to a lift of an end of the hyperbolic surface $\Sigma_{l}=\h^{2}/\rho_{l}(\pi_{1}(\Sigma))$. Let $\gamma \in \pi_{1}(\Sigma)$ be the associated peripheral element and let $c$ be the lift of $\gamma$ so that $C$ is a connected component of $\R\Pp^{1}\setminus \{c^{\pm}\}$ where $c^{\pm}$ are the end points of $c$. Similarly $c'^{\pm}=f_{\Gamma}(c^{\pm})$ are the ideal points of a lift $c'$ of a geodesic boundary of the hyperbolic surface $\Sigma_{r}=\h^{2}/\rho_{r}(\pi_{1}(\Sigma))$ and the oriented arc between $c'^{-}$ and $c'^{+}$ is a connected component of $\R\Pp^{1}\setminus \Lambda_{r}$. The pairs of points $(c^{+}, c'^{+})$ and $(c^{-}, c'^{-})$ belong to the limit set $\Lambda_{\rho}$ and are connected in $\Gamma$ by a $\rho(\gamma)$-equivariant light-like polygonal end $\{\ell_{i}\}_{i \in \Z}$. Without loss of generality we can assume that the segments $\ell_{2i}$ with even indices belong to the left-foliation and the segments $\ell_{2i+1}$ of odd indices lie in the right-foliation. The projections $\Pi_{l}(\ell_{2i+1})$ give a collection of arcs with end points $\vartheta_{i}=\Pi_{l}(\ell_{2i})$ limiting to $c^{-}$ when $i\to -\infty$ and to $c^{+}$ for $i\to +\infty$. Let us denote by $\vartheta'_{i}$ the projections $\Pi_{r}(\ell_{2i+1})$. We can extend the function $f_{\Gamma}$ to the connected component $C$ in such a way that $\vartheta_{i}$ are the points of discontinuity of $f_{\Gamma}$ and the following relations hold
\begin{align*}
	f_{\Gamma}(\mathrm{int}(\Pi_{l}(\ell_{2i+1})))&=\vartheta_{i}'  \\
	f_{\Gamma}(\vartheta_{i})&=\vartheta_{i}' \ \ \ \ \ \ \ \text{for all $i\in \Z$.}
\end{align*}
This determines the function $f_{\Gamma}$ in all connected components $\rho_{r}(\pi_{1}(\Sigma))C$ by equivariance, and repeating the same construction for all geodesic boundaries of $\Sigma_{l}$ we obtain the desired function $f_{\Gamma}$ whose graph is represented by the curve $\Gamma$. Notice that the function $f_{\Gamma}$ is uniquely determined by the representations $\rho_{l}$ and $\rho_{r}$, and by labelled collections of points $\{\vartheta_{i}\}_{i\in \Z}$ and $\{\vartheta_{i}'\}_{i\in \Z}$ in $\R\Pp^{1}$ for each geodesic boundary of the hyperbolic surfaces $\Sigma_{l}$ and $\Sigma_{r}$. 
Since this data can be obtained from the universal cover of two crowned hyperbolic surfaces, we get the following:

\begin{teo}\label{thm:parameterisation} The set $\mathcal{GH}^{\mathcal{LP}}(\Sigma, 2n_{1}-4, \dots, 2n_{N}-4)$ is parameterised by 
\[
	\left(\mathcal{T}(S, n_{1}-2, \dots, n_{N}-2)\times \mathcal{T}(S, n_{1}-2, \dots, n_{N}-2)\right)/\Z^{N}
\]
where $\Z^{N}$ acts diagonally by relabelling the lifts of the boundary cusps in the universal cover of each crown.
\end{teo}
\begin{proof}This is a consequence of the above discussion together with the remark that a diagonal change of labelling produces the same function, hence the same curve on $\partial_{\infty}\AdS_{3}$.
\end{proof}

\begin{cor}The map $\Psi: \mathcal{MQ}(\Sigma, n_{1}, \dots, n_{N}) \rightarrow \mathcal{GH}^{\mathcal{LP}}(\Sigma, 2n_{1}-4, \dots, 2n_{N}-4)$ is a homeomorphism.
\end{cor}
\begin{proof}We already know that $\Psi$ is a proper, injective and continous map. By the previous theorem, $\mathcal{GH}^{\mathcal{LP}}(\Sigma, 2n_{1}-4, \dots, 2n_{N}-4)$ is a manifold of dimension $2(6\tau-6+\sum_{i}(n_{i}+1))$ which equals the dimension of the subbundle $\mathcal{MQ}(\Sigma, n_{1}, \dots, n_{N})$, hence $\Psi$ is a homeomorphism.
\end{proof}

\bibliographystyle{alpha}
\bibliographystyle{ieeetr}
\bibliography{bs-bibliography}

\bigskip

\noindent \footnotesize \textsc{DEPARTMENT OF MATHEMATICS, RICE UNIVERSITY}\\
\emph{E-mail address:}  \verb|andrea_tamburelli@libero.it|

\end{document}